\documentclass[11pt, letterpaper, reqno]{amsart}

\usepackage{amsmath,amsfonts,amsthm,amssymb,amscd,mathtools}
\usepackage{mathrsfs, enumerate, accents}
\usepackage{latexsym}
\usepackage{tikz}
\usepackage{times}
\usepackage[T1]{fontenc}

\usepackage{chngcntr}

\counterwithin{figure}{section} %by-section numbering
\newtheorem{thm}{Theorem}[section]
\newtheorem*{thm*}{Theorem}
\newtheorem{lem}[thm]{Lemma}
\newtheorem{prop}[thm]{Proposition}

\theoremstyle{definition}
\newtheorem{defi}[thm]{Definition}

\theoremstyle{remark}
\newtheorem{rmk}[thm]{Remark}
\newtheorem*{rmk*}{Remark}

\numberwithin{equation}{section}

\newcommand{\RR}{\ensuremath{\mathbb{R}}}

\newcommand{\ZZ}{\ensuremath{\mathbb{Z}}}

\newcommand{\bJ}{{\bf J}}

\newcommand{\TT}{\mathbb{T}}
\newcommand{\bT}{{\bf T}}

\newcommand{\floor}[1]{\lfloor{#1}\rfloor}
\newcommand{\bfloor}[1]{\left\lfloor{#1}\right\rfloor}
\newcommand{\ceil}[1]{\lceil{#1}\rceil}
\newcommand{\bceil}[1]{\left\lceil{#1}\right\rceil}
\newcommand{\sfloor}[1]{\left\lfloor{#1}\right\rfloor^{'}}
\newcommand{\sceil}[1]{\left\lceil{#1}\right\rceil^{'}}

\newcommand{\uu}{{\mu'}}
\newcommand{\vv}{{\nu'}}

\newcommand{\cD}{\mathcal{D}} 
\newcommand{\cDprime}{{\mathcal{D}'}}

\newcommand{\cCprime}{{\mathcal{C}'}}
\newcommand{\cCprimebar}{\widetilde{\mathcal{C}}'}

\newcommand{\muu}{u}
\newcommand{\nuu}{v}
\newcommand{\egcd}{{\rm egcd}}
\newcommand{\elcm}{{\rm elcm}}
\newcommand{\alphap}{{\alpha'}}
\newcommand{\betap}{{\beta'}}
\newcommand{\tbetap}{{\sigma'}}

\newcommand{\trhop}{{\tau'}}
\newcommand{\talpha}{{\sigma}}
\newcommand{\trho}{{\tau}}
\newcommand{\Qthree}{\RR^2_{-}}
\newcommand{\mup}{{\mu'}}

\newcommand{\pS}{{S}^{'}}
\newcommand{\Si}{S_{(i)}}
\newcommand{\Sii}{S_{(ii)}}
\newcommand{\Siii}{S_{(iii^*)}}
\DeclareMathOperator{\lcm}{lcm}

\DeclareMathOperator{\rowspan}{row.span}
\definecolor {color_b}{RGB}{255,0,0}
\definecolor {color_c}{RGB}{20, 200, 30}
\definecolor {color_a}{RGB}{0,0,255}

\title[Dilated floor functions having nonnegative commutator II.]
{Dilated floor functions having nonnegative commutator\\
II. Negative dilations}

\author{J.\ C.\ Lagarias}
\address{Dept.\ of Mathematics, University of Michigan,
Ann Arbor, MI 48109--1043}
\email{lagarias@umich.edu}
\author{D.\ H.\ Richman}
\address{Dept.\ of Mathematics, University of Michigan,
Ann Arbor, MI 48109--1043}
\email{hrichman@umich.edu}


\thanks{Work of the  first author was partially supported by NSF grant DMS-1401224 and DMS-1701576,
by MSRI as a Chern Professor and by a Simons  Fellows in Mathematics grant. MSRI is partially supported
by an NSF grant. 
Work of the second author was partially supported by NSF grant DMS-1600223
and by a Rackham Predoctoral Fellowship.
}

\begin{document}
\begin{abstract}
This paper completes the classification of  the set $S$ of real parameter pairs $(\alpha,\beta)$  
such that the dilated floor functions 
$f_\alpha(x) = \floor{\alpha x}$ and $f_\beta(x) = \floor{\beta x}$ have a nonnegative commutator, i.e.
$ [ f_{\alpha}, f_{\beta}](x) = \floor{\alpha \floor{\beta x}} -  \floor{\beta \floor{\alpha  x}} \geq 0$
for all real $x$.
This paper treats the case where 
both dilation parameters $\alpha, \beta$ are negative.
This result is equivalent to 
classifying all positive $\alpha, \beta$ satisfying
$ \floor{\alpha \ceil{\beta x}} -  \floor{\beta \ceil{\alpha  x}} \geq 0$
for all real $x$.
The classification analysis  is connected with 
the theory of Beatty sequences and with the Diophantine Frobenius problem in two generators.
\end{abstract}

\maketitle

%%%%% uncomment to only show sections, no subsections
\setcounter{tocdepth}{1}
%\tableofcontents

\bibliographystyle{amsplain}

%%%%%%%%%%%%%%%%%%%%%%%%%%
% section  1:  Introduction
%%%%%%%%%%%%%%%%%%%%%%%%%%
\section{Introduction}
The {\em floor function}  
$\lfloor x \rfloor$
rounds a real number down to the nearest integer.
For  a real parameter $\alpha$,   the 
{\em dilated} floor function $f_\alpha(x)= \floor{\alpha x}$ performs discretization at the length scale $\alpha^{-1}$. 
Besides references noted in Part I, dilated floor functions are used  in describing
one-dimensional quasicrystals, see de Bruijn \cite[Sect. 4]{deB81}, \cite{deB89}, Brown \cite{Brown:1993} 
and Mingo \cite{Mingo00}.

This paper continues the study  of commutators
under composition of dilated floor functions 
$$ 
[f_{\alpha}, f_{\beta}](x) 
:= f_{\alpha} \circ f_{\beta}(x) - f_{\beta} \circ f_{\alpha}(x) 
= \floor{\alpha \floor{\beta x}} - \floor{\beta \floor{\alpha x}}.
$$
It  completes  the classification 
of the set $S$ of all pairs $(\alpha, \beta) \in \RR^2$ such that
\begin{equation}\label{eqn:ineq} 
\floor{\alpha \floor{\beta x}} \geq \floor{\beta \floor{\alpha x}} \quad\text{for all }x\in\RR,
\end{equation}
begun in Part I \cite{LagR:2018a}. 
That is, \eqref{eqn:ineq} says  $[f_{\alpha}, f_{\beta}](x) \ge 0$ for all real $x$, which we will often abbreviate 
as $[f_{\alpha}, f_{\beta}]\ge 0$,  omitting  the variable $x$.

\subsection{Classification theorem}

%%%%%%%%%%%%%%%%%%%
%
% Theorem 1.1 MAIN THEOREM Negative dilations 
%
%%%%%%%%%%%%%%%%%%%
\begin{thm}[Negative dilations classification]\label{thm:negative}
Given %dilation factors 
$\alpha<0,$  $ \beta < 0$, the inequality
$
 \floor{\alpha \floor{\beta x}} \geq \floor{\beta \floor{\alpha x}}
$
holds for all $x \in \RR$ if and only if one or more of the following conditions holds:

\begin{enumerate}[(i)]
\item[(i)]\label{it:quad3-a} 
There are integers $m \geq 0$, $n\geq 1$ such that
\begin{equation}\label{quad3-1}
 m \alpha \beta - n \beta  = -\alpha.
\end{equation}

\item[(ii)]
\label{it:quad3-b} 
There are coprime integers $p,q\geq 1$ such that
\begin{equation}\label{quad3-2}
\alpha =- \frac{q}{p}, \quad  - \frac{1}{p} \leq \beta < 0 .
\end{equation}

\item[$(iii^{\ast})$]
 \label{it:quad3-c} There are coprime integers $p, q \geq 1$  and 
 %integers 
  $m\geq 0, n\geq 1, r \geq 1$ such that
\begin{equation}\label{quad3-3}
\alpha = -\frac{q}{p}, \quad \beta = -\frac{1}{p}\left(1 + \frac{1}{r}\left( \frac{m}{p} + \frac{n}{q}-1\right) \right)^{-1}.
\end{equation}
\end{enumerate}
\end{thm}

In this classification,
$(iii^{\ast})$  is a strictly larger set than   $(iii)$   in Theorem 1.3 stated in  Part I;  
however the two theorems are equivalent. 
Case $(iii^{\ast})$ differs from $(iii)$ in  allowing value  $r=1$ and in not  imposing the side condition 
$0<\frac{m}{p} + \frac{n}{q} <1$  in $(iii)$.  
However, all  members of case $(iii^{\ast})$ coming from $r=1$  also  appear in case $(i)$,
and all  members 
 that have  $ \frac{m}{p} + \frac{n}{q} \ge 1$ also  appear in  case $(ii)$.
Condition $(iii^*)$ parameterizes  
{\em all} negative rational solutions in $S$, 
see Theorem \ref{thm:neg-necessary}.
The parametrization in  $(iii^{\ast})$ is  redundant:
%in the sense that 
some $(\alpha, \beta)$ values arise from
multiple  parameter values  $(p,q,m,n,r)$.

Figure \ref{fig:a-b-prime-coord} pictures the solution set $S$
for negative dilations
viewed in the the positive $(\alphap, \betap)$-coordinate system, letting $\alphap = -\alpha, \betap= -\beta$;
the set pictured is   $S' := -S$ in $(\alphap, \betap)$ coordinates.

%%%%%%%%%%%%%%%%%%%%%%%%%%%%%
% Figure 1.2  (alphap, \betap) coordinates. quadrant 3 
%%%%%%%%%%%%%%%%%%%%%%%%%%%%%%%%%%%%%%%%%%%%%%%%%%%%%%%%%%%%%%%%%%%%%%%%%%%%%%%%%%%
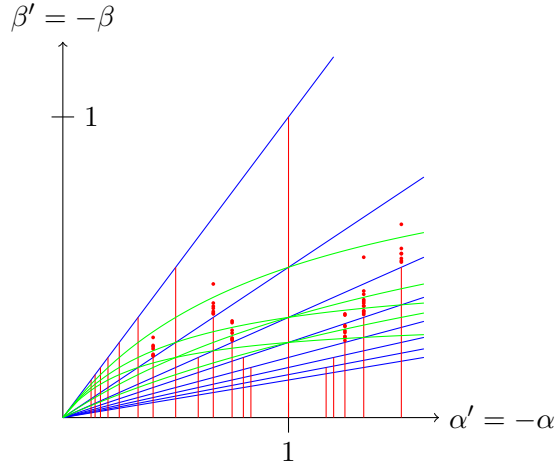
\begin{figure}[h]
\begin{center}
\begin{tikzpicture}[scale=0.50]
  \draw[->] (0,0) -- (10,0) node[right] {$\alpha' = -\alpha$};
  \draw[->] (0,0) -- (0,10) node[above] {$\beta' = -\beta$};
  \draw[xscale=6,yscale=8] (1,0.05) -- (1,-0.05) node[below] {$1$};
  \draw[xscale=6,yscale=8] (-0.05,1) -- (0.05,1) node[right] {$1$};
    %\draw[xscale=8,domain=0:1.2,smooth,variable=\x, blue] plot ({\x},{1});
  \draw[xscale=6,yscale=8,domain=1.2:0,smooth,variable=\x, blue] plot ({\x},{\x});
  \draw[xscale=6,yscale=8,domain=1.6:0,smooth,variable=\x, blue] plot ({\x},{\x/2});
  \draw[xscale=6,yscale=8,domain=1.6:0,smooth,variable=\x, blue] plot ({\x},{\x/3});
  \draw[xscale=6,yscale=8,domain=1.6:0,smooth,variable=\x, blue] plot ({\x},{\x/4});
  \draw[xscale=6,yscale=8,domain=1.6:0,smooth,variable=\x, blue] plot ({\x},{\x/5});
  \draw[xscale=6,yscale=8,domain=1.6:0,smooth,variable=\x, blue] plot ({\x},{\x/6});
  \draw[xscale=6,yscale=8,domain=1.6:0,smooth,variable=\x, blue] plot ({\x},{\x/7});
  \draw[xscale=6,yscale=8,domain=1.6:0,smooth,variable=\x, blue] plot ({\x},{\x/8});
  
  \draw[xscale=6,yscale=8,domain=1:0,smooth,variable=\y, red]  plot ({1},{\y});
  \draw[xscale=6,yscale=8,domain=1/2:0,smooth,variable=\y, red]  plot ({1/2},{\y});
  \draw[xscale=6,yscale=8,domain=1/2:0,smooth,variable=\y, red]  plot ({3/2},{\y});
  \draw[xscale=6,yscale=8,domain=1/3:0,smooth,variable=\y, red]  plot ({1/3},{\y});
  \draw[xscale=6,yscale=8,domain=1/3:0,smooth,variable=\y, red]  plot ({2/3},{\y});
  \draw[xscale=6,yscale=8,domain=1/3:0,smooth,variable=\y, red]  plot ({4/3},{\y});
  \draw[xscale=6,yscale=8,domain=1/4:0,smooth,variable=\y, red]  plot ({1/4},{\y});
  \draw[xscale=6,yscale=8,domain=1/4:0,smooth,variable=\y, red]  plot ({3/4},{\y});
  \draw[xscale=6,yscale=8,domain=1/4:0,smooth,variable=\y, red]  plot ({5/4},{\y});
  \draw[xscale=6,yscale=8,domain=1/5:0,smooth,variable=\y, red]  plot ({1/5},{\y});
  \draw[xscale=6,yscale=8,domain=1/5:0,smooth,variable=\y, red]  plot ({2/5},{\y});
  \draw[xscale=6,yscale=8,domain=1/5:0,smooth,variable=\y, red]  plot ({3/5},{\y});
  \draw[xscale=6,yscale=8,domain=1/5:0,smooth,variable=\y, red]  plot ({4/5},{\y});
  \draw[xscale=6,yscale=8,domain=1/5:0,smooth,variable=\y, red]  plot ({6/5},{\y});
  \draw[xscale=6,yscale=8,domain=1/6:0,smooth,variable=\y, red]  plot ({1/6},{\y});
  \draw[xscale=6,yscale=8,domain=1/6:0,smooth,variable=\y, red]  plot ({5/6},{\y});
  \draw[xscale=6,yscale=8,domain=1/6:0,smooth,variable=\y, red]  plot ({7/6},{\y});
  \draw[xscale=6,yscale=8,domain=1/7:0,smooth,variable=\y, red]  plot ({1/7},{\y});
  \draw[xscale=6,yscale=8,domain=1/8:0,smooth,variable=\y, red]  plot ({1/8},{\y});
  
  % sporadic solutions
  \foreach \y in {     9/14   % a = 3/2,  b = 3r/(6r-4)
                 ,     9/16,12/22   %     b = 3r/(6r-2)
                 ,6/11,9/17,12/23,15/29   %     b = 3r/(6r-1)
    }{ \node[circle,inner sep=0.5pt,fill,red] at (3/2*6,1*\y*8) {};  }
  \foreach \y in {4/9,      8/21,10/27 % a = 2/3, b = 2r/(6r-3)
                 ,4/11,6/17,8/23,10/29  %         b = 2r/(6r-1)
    }{ \node[circle,inner sep=0.5pt,fill,red] at (2/3*6,1*\y*8) {};  }
  \foreach \y in {8/15,      16/39,20/51   % a = 4/3, b = 4r/(12r-9)
                                  ,20/54    %       b = 4r/(12r-6)
                 ,8/19,12/31,16/43          %       b = 4r/(12r-5)
                 ,8/21,      16/45          %       b = 4r/(12r-3)
                 ,     12/34                %       b = 4r/(12r-2)
                 ,8/23,12/35                %       b = 4r/(12r-1) 
    }{ \node[circle,inner sep=0.5pt,fill,red] at (4/3*6,1*\y*8) {};  }
  \foreach \y in {4/15,6/25,8/35  % a = 2/5, b = 2r/(10r-5)
                 ,4/17               %       b = 2r/(10r-3)
                 ,4/19,6/29,8/39     %       b = 2r/(10r-1)
    }{ \node[circle,inner sep=0.5pt,fill,red] at (2/5*6,1*\y*8) {};  }
  \foreach \y in {   9/28  % a = 3/4, b = 3r/(12r-8)
                 ,6/19,9/31   %       b = 3r/(12r-5)
                 ,     9/32   %       b = 3r/(12r-4)
                 ,     9/34   %       b = 3r/(12r-2)
                 ,6/23,9/35   %       b = 3r/(12r-1)
    }{ \node[circle,inner sep=0.5pt,fill,red] at (3/4*6,1*\y*8) {};  }
  \foreach \y in {   15/44  % a = 5/4, b = 5r/(20r-16)
                               %       b = 5r/(20r-12)
                 ,10/29,15/49  %       b = 5r/(20r-11)
                 ,      15/52  %       b = 5r/(20r-8)
                 ,10/33,15/53  %       b = 5r/(20r-7)
                               %       b = 5r/(20r-6)
                 ,      15/56  %       b = 5r/(20r-4)
                 ,10/37        %       b = 5r/(20r-3)
                               %       b = 5r/(20r-2)
                 ,10/39        %       b = 5r/(20r-1) 
    }{ \node[circle,inner sep=0.5pt,fill,red] at (5/4*6,1*\y*8) {};  }
  
  \draw[xscale=6,yscale=8,domain=1.6:0,smooth,variable=\x, green] plot ({\x},{\x/(1+\x)});
  \draw[xscale=6,yscale=8,domain=1.6:0,smooth,variable=\x, green] plot ({\x},{\x/(1+2*\x)});
  \draw[xscale=6,yscale=8,domain=1.6:0,smooth,variable=\x, green] plot ({\x},{\x/(1+3*\x)});
  \draw[xscale=6,yscale=8,domain=1.6:0,smooth,variable=\x, green] plot ({\x},{\x/(2+\x)});
  \draw[xscale=6,yscale=8,domain=1.6:0,smooth,variable=\x, green] plot ({\x},{\x/(3+\x)});
\end{tikzpicture}
\end{center}
\caption{Negative dilation solutions 
 $S'=-S$ viewed in $(\alphap, \betap)$-coordinates. Sporadic solutions appear as dots.}
%with $\uu= -1/\beta$ and $\vv= \frac{\alpha}{\beta}$. }
\label{fig:a-b-prime-coord}
\end{figure} 

Figure \ref{fig:a-b-prime-coord} exhibits
two kinds of solutions having no analogue in the
positive dilation case.
The first consists of  vertical line segments 
(of variable finite length) for all rational values of $\alpha$, 
given by case $(ii)$ of Theorem   \ref{thm:negative}. 
%These $\alpha$-values are dense on the 
%negative real axis.
The second consists of  a countable set of %$0$-dimensional components, 
isolated points having rational coordinates, which we term ``sporadic rational solutions.''
They appear as dots above the vertical line segments in Figure \ref{fig:a-b-prime-coord}.
They form  a strict subset of  case $(iii^{\ast})$;
determining this subset involves removing from case $(iii^{\ast})$  
the overlap  shared with case $(i)$ and  $(ii)$  solutions, 
and removing further redundancy from the parametrization in $(iii^*)$.
This determination 
%problem of uniquely listing all sporadic rational solutions 
is related to 
the Diophantine Frobenius problem in the parameters $(p,q)$, see \cite{RA05}. 
We treat this problem elsewhere  (\cite{LagR:2019c}).

%We note that commutators  of dilated floor functions are always bounded step functions, which need not be  periodic.
%They  are examples of  bounded generalized polynomials in the sense of Bergelson and Leibman \cite{BerL07}.
%Dilated  floor functions combined with translation satisfy many non-obvious
%identities, see Graham, Knuth and Patashnik for \cite[Chap. 3]{GKP94}. 

%%%%%%%%%%%%%%%%%%%
%
% SUBSECTION  1.2 Closure Property 
%
%%%%%%%%%%%%%%%%%%%

\subsection{The set  $S$ is closed} 

In Part~I we showed that the intersection of $S$ with each of the closed first, 
second and fourth quadrants of the $(\alpha, \beta)$-plane is closed. 
To complete the proof  that $S$ is closed, we establish in this paper that the  intersection 
with the  closed third quadrant is closed.

%%%%%%%%%%%%%%%%%%%
% Theorem 1.2 Closed subset
%%%%%%%%%%%%%%%%%%%
\begin{thm}
[Closed set property of $S$] % for nonpositive dilations]
\label{thm:closed}
The set of all pairs of dilation factors $(\alpha, \beta)$ 
with $\alpha \le 0, \beta \le 0$ satisfying  the nonnegative commutator relation 
${[f_\alpha, f_\beta](x) \geq 0 }$
is a closed subset of $\RR^2$.
%is closed.
\end{thm}

\noindent 
%{As remarked in Part~I, 
%The  closed set property seems not obvious 
%because 
%since 
The function 
$[f_\alpha, f_\beta](x) = \floor{\alpha\floor{\beta x}} - \floor{\beta\floor{\alpha x}}$ 
is  discontinuous in each of the variables $\alpha, \beta, x$,
so the closure property is not obvious.
% the closed set property seems not obvious.
We establish this result in Section \ref{sec:consequences} 
using the  classification of the  set $S$  in Theorem \ref{thm:negative}.

%%%%%%%%%%%%%%%%%%%%%%%%%%%%%%%%%%
%
% SUBSECTION  1.4    Symmetries of S on positive and negative
%
%%%%%%%%%%%%%%%%%%%%%%%%%%%%%%%%%%

\subsection{Symmetries of parameter set  $S$: negative dilations} 
\label{subsec:symm-neg}
%In the process of proving the classification theorem we 
We establish several symmetries of the set $S$ in the negative  dilation case.
There are a set of linear symmetries paralleling the positive dilation case. 

%%%%%%%%%%%%%%%%%%%
%
% Theorem 1.4
%  
%
%%%%%%%%%%%%%%%%%%%
\begin{thm}[Symmetries of $S$: negative dilations]
\label{thm:symmetries-n}
For negative dilations $\alpha, \beta < 0$, 
the  set $S$ is mapped into itself 
under the following symmetries:
\begin{enumerate}[(i)]
\item For any integer $m \ge 1$, if $(\alpha, \beta) \in S$, then $(m\alpha, \beta) \in S$.
\item For any integer $m \ge 1$, if $(\alpha, \beta) \in S$, then 
$(\frac{1}{m}\alpha, \frac{1}{m}\beta) \in S$.
\end{enumerate}
\end{thm}

An important new feature of the  negative dilation case is the existence of
additional symmetries having no positive dilation analogue. 
These symmetries act on vertical lines 
% $S_{\alpha}= \{ (\alpha, \beta)  \in S\}$ restricted to
with fixed rational  $\alpha$-coordinate.
When $\alpha= -\frac{q}{p}$ (in lowest terms),  
the symmetries act on the $\beta$-coordinate  as special cases of 
the family of linear fractional transformations
 \begin{equation*}
\Phi_{p}^r(\beta) := \frac{r \beta}{ (1-r)p \beta +1}  
\end{equation*}
defined for all real $r>0$, which depend on $p$ and not on  $q$.
%We use only integer parameters $r>0$.

%%%%%%%%%%%%%%%
% Theorem 1.5
%%%%%%%%%%%%%%%
\begin{thm}[Linear fractional symmetries]
\label{thm:lin-frac-sym}
Let  $\alpha = -\frac{q}{p}$ be a fixed negative rational. 
For each integer $r \ge 1$, 
the  linear fractional transformation
\begin{equation*}
% \label{eqn:lin-frac}
\Phi_{p}^r(\beta) := \frac{r \beta}{ (1-r)p \beta +1}  
\end{equation*}
maps the  set of  $\beta<0$ which satisfy 
$[f_{-q/p}, f_{\beta}] \geq 0$
into itself. 
Namely,
if $(-\frac{q}{p}, \beta) \in S$
then $(-\frac{q}{p}, \Phi_p^r(\beta)) \in S$.
\end{thm}

%The linear fractional transformation $\Phi_p^r$ can represented 
%in one higher dimension by the matrix
%$\begin{pmatrix}
%r & 0\\ (1-r)p & 1
%\end{pmatrix}$
%of determinant $r$..

The set of  transformations $\Phi_{p}^r$ 
with $p$ fixed and integer  $r\geq 1$ forms a commutative semigroup
under composition of functions:
\begin{equation*} 
\Phi_{p}^{r} \circ \Phi_{p}^{s} = \Phi_{p}^{rs} .
\end{equation*}
All elements of this semigroup have two common  fixed points  $z=0$ and $z= -\frac{1}{p}$. 
When  $r= 1$ the  map  is the identity.
\begin{enumerate}
\item
The map $\Phi_{p}^r$ is a  homeomorphism of the open interval $( -\frac{1}{p},0) \to ( -\frac{1}{p},0)$.
For $\alpha= -\frac{q}{p}$ this  interval   
is exactly  the $\beta$ values in Case $(ii)$.
% solutions, which are permuted under these actions.
%continuous family  Case (ii) solutions having $\alpha= -\frac{q}{p}$. 
\item 
The map $\Phi_{p}^r$ 
%give  homeomorphisms of 
sends the open interval  $(-\infty, -\frac{1}{p})$ % \to (-\infty,  -\frac{1}{p})$;
to 
%for $r \geq 2$  the range is 
$(-\frac{r}{(r-1)p}, -\frac{1}{p})$, for $r\geq 2$.
%the map $\Phi_{p}^r$ sends the interval $(-\infty, 0)$ to the interval $(-\frac{r}{(r-1)p},0)$.
For  $\alpha= -\frac{q}{p}$ the  case $(iii^{\ast})$ solutions with $\beta\in (-\infty, -\frac{1}{p} )$
are the ones with  $0 < \frac{m}{p} + \frac{n}{q} <1$, and are a discrete set. 
The action of $\Phi_{p}^r$ on these solutions is injective, but for $r\ge 2$ usually not surjective. 
\end{enumerate}

Theorem \ref{thm:lin-frac-sym} is the mechanism used  in this paper to 
show existence of  all the 
% constructing all the
case $(iii^{\ast})$  solutions. 
We start from the case~$(i)$ rational solutions $(-\frac{q}{p}, \beta)$
having $\beta < -\frac{1}{p}$; these solutions 
%depend on parameters $m,n$,
%and 
correspond to  $r=1$ in  the parameterization in case $(iii^{\ast})$.
From these solutions for fixed $\alpha= -\frac{q}{p}$ 
the full infinite set of sporadic rational solutions are obtained 
by applying the linear fractional maps $\Phi_{p}^r$ for $r\ge 2$ to  these solutions.
We prove Theorem \ref{thm:lin-frac-sym} in Section \ref{subsubsec:integer-contact-proof}.
%These monotonically decrease the value of $\betap= |\beta|$,
%see Section~\ref{subsec:expansion-contraction}. 

%%%%%%%%%%%%%%%%%%%%%%%%%%%%%%%%%%%%%%
%
% SECTION 2: Outline of proofs 
%
%%%%%%%%%%%%%%%%%%%%%%%%%%%%%%%%%%%%%%
\section{Outline of proofs}\label{subsec:24}

The negative dilation case analysis is similar 
to the positive dilation case, but is  more complicated.
Part I noted an analogy with  the classification  of disjoint
Beatty sequences; a similar connection appears here in 
Lemma \ref{lem:diagonal-vert-boundary}.
%The most important difference is the presence of extra symmetries
%An important difference is the proof of the s Theorem ~\ref{thm:lin-frac-symm}. 
% for rational $\alpha$. 

% SECTION 3
Section \ref{sec:rounding-S}  
introduces the family of  {\em strict rounding functions},
whose definition parallels that  of rounding functions in Part~I,
replacing the rescaled floor function with 
rescaled versions of the {\em strict floor function}.
Proposition  \ref{prop:rounding-neg} gives a criterion for
\eqref{eqn:ineq} to hold in terms of strict rounding functions.

%SECTION 4
Section~\ref{sec:negative-LFS} establishes 
that $S$ is closed under the linear fractional
transformations given in Theorem \ref{thm:lin-frac-sym}
(for rational $\alpha$).  These symmetries provide a mechanism to construct
all type $(iii^{\ast})$ solutions starting from type $(i)$ rational solutions.
This proof proceeds  in a new coordinate system, the $(\uu, \vv)$-coordinates, 
given by $\uu= - \frac{1}{\beta},\, \vv=\frac{\alpha}{\beta}$, 
which take positive values.

% SECTION 5
Section~\ref{sec:negative-S}  establishes the sufficiency direction of
Theorem~\ref{thm:negative}. 
It shows all parameters given in cases $(i), (ii), (iii^{\ast})$
of Theorem \ref{thm:negative} are solutions in $S$. 
It uses $(\uu, \vv)$-coordinates.

%SECTION 6
 Section \ref{sec:negative-N} establishes the necessity direction of
Theorem~\ref{thm:negative}. Its main result  
partitions all negative dilation solutions 
into  two irrational cases $(i^{\ast})$ and $(ii^{\ast})$  and the rational 
case $(iii^{\ast})$ which it shows 
enumerates all the rational points in $S$.
It uses $(\tbetap, \trhop)$-coordinates,
with $\tbetap= \frac{1}{\uu}$, $\trhop=\frac{1}{{\vv}}$. 

% SECTION 7
Section \ref{sec:consequences}  proves that the part of 
$S$  restricted to the region of nonpositive dilation factors is a closed set.
This result  completes the proof that $S$ is closed in  $\RR^2$,  stated in Part I \cite[Theorem 1.4]{LagR:2018a}.

%%%%%%%%%%%%%%%%%%%%%%%%%%%%%%%%%%
%
% SUBSECTION  2.2
%%%%%%%%%%%%%%%%%%%%%%%%%%%%%%%%%%

\subsection{Notational conventions} 
\label{subsec:notation}
(1) The moduli space parametrizing  dilations 
 has  coordinates $(\alpha, \beta)$.
 The proofs use different coordinate systems for this moduli space, starting with 
 the negation $(\alphap, \betap)$
of the  coordinate system $(\alpha, \beta)$. Thus $\pS= -S$ denotes the solution set $S$ in $(\alphap, \betap)$-coordinates.
In Section \ref{sec:negative-S} we use the $(\uu, \vv)$-coordinate systems
and in Section \ref{sec:negative-N} the
$(\tbetap, \trhop)$-coordinate system. 
These  coordinate systems take  positive real values.

(2) Variables  $x$ and $y$ are never moduli space coordinates.   
They are real-valued, and  are used  in   
 dilated floor functions, in commutators $[f_{\alpha}, f_{\beta}](x)$,
 and  in  lattices and quotient tori.

%%%%%%%%%%%%%%%%%%%%%%%%%%%%%%%%%%%%%%
%
% Section 3  strict rounding functions
%
%%%%%%%%%%%%%%%%%%%%%%%%%%%%%%%%%%%%%%
\section{Strict rounding functions}
\label{sec:rounding-S}

In Part~I, rounding  functions
$\floor{x}_\alpha := \alpha \bfloor{\frac{1}{\alpha} x} $ 
were used for analyzing the relation $[f_\alpha, f_\beta] \geq 0$
in the positive dilation case.
Here we use  modified  rounding  functions, obtained by
replacing the floor function with 
the {\em strict floor function} $\sfloor{x}$, which  
 returns the largest integer strictly smaller than $x$. 
Thus 
$ 
\sfloor{x} := \max\{n\in \ZZ : n < x\},
$
which also has
$ \sfloor{x} = \ceil{x}-1$.
%%%%%%%%%%%%%%%%%%%%%%%
% Defn  3.1
% Strict rounding functions
%%%%%%%%%%%%%%%%%%%%%%%
\begin{defi}\label{defi:strict-round}
Given  nonzero $\alpha$, 
let $\sfloor{x}_{\alpha}$ 
denote the {\em (lower) strict  rounding function}  defined by
\begin{equation*}
\sfloor{x}_{\alpha} := \alpha \sfloor{\frac{1}{\alpha} x} 
= \alpha \left(\bceil{\frac{1}{\alpha}x} - 1\right) ,
\end{equation*}
We additionally set $\sfloor{x}_0= x$.
%%%%%%%%%%%%%%%%%%%%%%%%%%%%%%%%%%%%
%Here $\sceil{x}_{\alpha}  = \sfloor{x}_{- \alpha}$. 
%%%%%%%%%%%%%%%%%%%%%%%%%%%%%%%%%%%% 

\end{defi}

We may also define the {\em strict ceiling function} $\sceil{x}= - \sfloor{-x} =  \bfloor{x} +1$
and the 
 {\em (upper) strict  rounding function} defined by
$\sceil{x}_{\alpha} := \alpha \sceil{\frac{1}{\alpha} x}$.
These definitions yield a  single family of 
strict rounding functions, because
\begin{equation}\label{eqn:strict-flip}
\sceil{x}_{\alpha} = \sfloor{x}_{-\alpha} 
\end{equation}
for all $\alpha \ne 0$.  
The continuous function   $\sfloor{x}_0= \sceil{x}_0=x$ 
arises as the pointwise
limit of the functions $\sfloor{x}_{\alpha}$ as $\alpha \to 0$.

%%%%%%%%%%%%%%%%%%%%%%%%%
% Subsection 3.1 Strict rounding:ORDERING
%%%%%%%%%%%%%%%%%:%%%%%%%%
\subsection{Strict rounding functions: ordering inequalities} 
\label{subsec:strict}

We  classify when the graph of one
strict rounding function  lies (weakly) below the other.

%%%%%%%%%%%%%%%%%%%
% Proposition 3.2 Strict Rounding functions
%%%%%%%%%%%%%%%%%%%
\begin{prop}[Strict  Rounding Function: Ordering Inequalities]
\label{prop:strict-rounding-relations}
Given positive $\alpha, \beta$,  the strict rounding functions satisfy the inequalities
\begin{enumerate}[(a)]
\item   $\sfloor{x}_{\alpha} \leq \sfloor{x}_{\beta}$ holds for all $x\in\RR$ if and only if 
$\alpha= m \beta$ for some integer $m\geq 1$.

\item 
$\sceil{x}_{\alpha} \leq \sceil{x}_{\beta}$  holds for all $x\in\RR$ if and only if 
$\beta = m\alpha$ for some integer $m\geq 1$.
\end{enumerate}
\end{prop}

\begin{proof}
This result parallels
\cite[Prop. 4.2]{LagR:2018a},
 making  use of the following general observation: if functions $f(x)$ and $g(x)$ are 
both piecewise continuous and left-continuous (resp. both right-continuous), then
$f \leq g$ holds identically
 if and only if it holds on the interior of the domain of continuity of $f$ and $g$.
Here  the  strict floor function $\sfloor{x}_{ \alpha}$ (resp. $\sceil{x}_{ \alpha}$) agrees with 
$\floor{x}_{\alpha}$ (resp. $\ceil{x}_{\alpha}$) on the interior of its domain of continuity.
\end{proof}

%%%%%%%%%%%%%%%%%%%%%%%%%
%
% Subsection 3.2 Strict Rounding Function Criterion
%
%%%%%%%%%%%%%%%%%%%%%%%%%
\subsection{Strict rounding function criterion: negative dilations}
\label{sec:nge-rounding-crit}

We  give a (separated variable) criterion in terms of strict rounding functions  
equivalent to the nonnegative commutator condition $[f_\alpha,f_\beta]\geq 0$
on dilated floor functions with negative dilation factors. We state it in 
$(\alphap, \betap) := (-\alpha, -\beta)$ coordinates.

%%%%%%%%%%%%%%%%%%%
% Proposition 3.3 
%%%%%%%%%%%%%%%%%%%
\begin{prop}[Nonnegative commutator relation: strict rounding function criterion]
\label{prop:rounding-neg}
Given
${\alphap, \betap > 0}$, the following properties are equivalent.
 \begin{enumerate}
 \item[\textnormal{(R1')}]
 The nonnegative commutator relation
 $[f_{-\alphap}, f_{-\betap}]\geq 0$ holds.
 \item[\textnormal{(R2')}]
(Lower strict rounding function) There holds
\begin{equation}\label{rounding-negative}
\sfloor{n}_{\alphap} \leq \sfloor{n}_{\betap} \quad\text{for all }n\in \ZZ,
\end{equation}
where $\sfloor{x}_\alphap $
is the  strict rounding function with parameter $\alphap$. 
\end{enumerate}
\end{prop}

\begin{proof}
The  proof parallels that of Proposition 4.3 in  \cite[p. 284]{LagR:2018a}.
The fundamental observation is that
\begin{equation}
\label{eqn:level3}
\{x:  f_{-\alphap}\circ f_{-\betap}(x) \ge n\}
 = \{ x : x > \frac1{\alphap\betap} \sfloor{n}_{\alphap} \}  .
\end{equation}
We leave the details to the interested reader.
%This formula follows from the chain of equivalences:
%  \begin{align*}
%  \floor{ -\alphap\floor{- \betap x }} \ge n 
%    &\quad\Leftrightarrow\quad {-\alphap} \floor{ {-\betap} x} \ge n 
%      && \mbox{(the right side is in $\ZZ$)}\\
%    &\quad\Leftrightarrow\quad \floor{ {-\betap} x } \le -\frac{1}{\alphap} n 
%      && \mbox{(since $-\alphap < 0$)}\\
%    &\quad\Leftrightarrow\quad \floor{ {-\betap} x } \le \bfloor{  -\frac{1}{\alphap} n } 
%      && \mbox{(the left side is in $\ZZ$)} \\
%    &\quad\Leftrightarrow\quad {-\betap} x < \bfloor{ - \frac{1}{\alphap}n } + 1 
%      &&  \text{(the right side is in $\ZZ$)}\\
%    &\quad\Leftrightarrow\quad {-\betap} x < -\bceil{ \frac{1}{\alphap}n } + 1 
%      &&  \text{(since $\floor{-x} = -\ceil{x}$)}\\
%    &\quad\Leftrightarrow\quad \alphap \betap x > \alphap \left(\bceil{ \frac{1}{\alphap} n} -1\right)
%     && \mbox{(multiply  by $-\alphap < 0$)}\\
%    &\quad\Leftrightarrow\quad \alphap \betap x > % \sceil{n}_{-\alphap} = 
%                                                   \sfloor{n}_{\alphap}  
%      && \mbox{(definition of strict rounding function)}.
%  \end{align*}
%The level set inclusion \eqref{eqn:level2} is equivalent by \eqref{eqn:level3} to  \eqref{rounding-negative}.
%\[
%\sfloor{n}_{\alphap}  \leq   \sfloor{n}_{\betap}
%  \quad \mbox{for all } n\in \ZZ ,
%  \]
%as asserted.
\end{proof}

%%%%%%%%%%%%%%%%%%%%%%%%%
%
% Subsection 3.3
%
%%%%%%%%%%%%%%%%%%%%%%%%%
\subsection{Symmetries of $S$ for negative dilations: Proof of Theorem \ref{thm:symmetries-n}} \label{subsec:symmetries-third}

%%%%%%%%%%%%%%%%%%%
% Proof of Theorem 1.4
%%%%%%%%%%%%%%%%%%%%
\begin{proof}[Proof of Theorem \ref{thm:symmetries-n}]
We suppose $\alphap, \betap > 0$ and are to show:
\begin{enumerate}
\item[(i)]  for any integer $m\geq 1$,  if $(-\alphap, -\betap) \in S$  then $( -\alphap, -{m}\betap)\in S$.
\item[(ii)]  for any integer $m\geq 1$,  if $(-\alphap, -\betap) \in S$  then 
$(-\frac{1}{m}{\alphap}, -\frac{1}{m}\betap) \in S$.
\end{enumerate}
These follow from Proposition \ref{prop:strict-rounding-relations}
and Proposition \ref{prop:rounding-neg};
the proof parallels that of Theorem 2.1 in \cite[p.~285]{LagR:2018a}.
\end{proof}

%%%%%%%%%%%%%%%%%%%%%%%%%%%%%%%%%%%%%%%%%%%%%%%%%%%%%%%%%%%%%%%%%%%%%%%%%%%%%%%%%%%
%
% SECTION  4:Linear fractional symmetries 
%
%%%%%%%%%%%%%%%%%%%%%%%%%%%%%%%%%%%%%%%%%%%%%%%%%%%%%%%%%%%%%%%%%%%%%%%%%%%%%%%%%%%
\section{Negative dilations classification: linear fractional symmetries}
\label{sec:negative-LFS}

This section   establishes the linear fractional symmetries of $S$ for negative dilations.
We will use these symmetries to generate all the rational solutions in case  $(iii^{\ast})$,
in particular to produce the sporadic rational solutions.

%%%%%%%%%%%%%%%%%%%%%%%%%%%%%%%%%
%
% SUBSECTION N 4.1 Birational coordinate changes
%
%%%%%%%%%%%%%%%%%%%%%%%%%%%%%%%%%
\subsection{Birational coordinate change: $(\uu,\vv)$-coordinates} \label{sec:birational-1} 

The proofs use  a birational change
of coordinates of the parameter space describing the two
negative dilations.
We map  $(\alphap, \betap)$-coordinates of the parameter space to
$(\uu, \vv)$-coordinates, given by 
\begin{equation}\label{eqn:uv-prime-coord}
(\uu,\vv) := \left( \frac{1}{\betap}, \frac{\alphap}{\betap} \right) = \left( -\frac{1}{\beta}, \frac{\alpha}{\beta} \right).
\end{equation}
Its inverse map is  
$(\alphap, \betap) = \left( \frac{\vv}{\uu}, \frac{1}{\uu} \right) $.
The negative dilation part of the solution set $S$ 
is pictured  in the new coordinates  in 
Figure \ref{fig:uv-prime-coord}.

%%%%%%%%%%%%%%%%%%%%%%%%%%%%%%
%
%  Figure N4.1 S in (\uu, \vv) coordinates 
%
%%%%%%%%%%%%%%%%%%%%%%%%%%%%%%%
\begin{figure}[h]
\begin{center}
  \begin{tikzpicture}[scale=.50]
  \draw[->] (0,0) -- (10,0) node[right] {$\uu= 1/\betap$};
  \draw[->] (0,0) -- (0,10) node[above] {$\vv= \alphap/\betap$};
  \draw[scale=1.8] (1,0.1) -- (1,-0.1) node[below] {1};
  \draw[scale=1.8] (0.1,1) -- (-0.1,1) node[left] {1};
  \begin{scope}[scale=1.8,blue]
    \draw[domain=0:5.2,smooth,variable=\x] plot ({\x},{1});
    \draw[domain=0:5.2,smooth,variable=\x] plot ({\x},{2});
    \draw[domain=0:5.2,smooth,variable=\x] plot ({\x},{3});
    \draw[domain=0:5.2,smooth,variable=\x] plot ({\x},{4});
    \draw[domain=0:5.2,smooth,variable=\x] plot ({\x},{5});
  \end{scope}
  \begin{scope}[scale=1.8,red]
    \draw[domain=1:5.2,smooth,variable=\y]  plot ({\y},{\y});
    \draw[domain=1:2.6,smooth,variable=\y]  plot ({\y},{2*\y});
    \draw[domain=1:1.75,smooth,variable=\y]  plot ({\y},{3*\y});
    \draw[domain=1:1.3,smooth,variable=\y]  plot ({\y},{4*\y});
    \draw[domain=1:1.05,smooth,variable=\y]  plot ({\y},{5*\y});
    \draw[domain=1:2.6,smooth,variable=\y]  plot ({2*\y},{\y});
    \draw[domain=1:1.75,smooth,variable=\y]  plot ({3*\y},{\y});
    \draw[domain=1:1.3,smooth,variable=\y]  plot ({4*\y},{\y});
    \draw[domain=1:1.05,smooth,variable=\y]  plot ({5*\y},{\y});
    \draw[domain=1:1.7,smooth,variable=\y]  plot ({2*\y},{3*\y});
    \draw[domain=1:1.8,smooth,variable=\y]  plot ({3*\y},{2*\y});
    \draw[domain=1:1.3,smooth,variable=\y]  plot ({4*\y},{3*\y});
    \draw[domain=1:1.3,smooth,variable=\y]  plot ({3*\y},{4*\y});
    \draw[domain=1:1.05,smooth,variable=\y]  plot ({3*\y},{5*\y});
    \draw[domain=1:1.05,smooth,variable=\y]  plot ({4*\y},{5*\y});
    \draw[domain=1:1.05,smooth,variable=\y, red]  plot ({5*\y},{3*\y});
    \draw[domain=1:1.05,smooth,variable=\y]  plot ({5*\y},{4*\y});
  \end{scope}
  % sporadic solutions
  \begin{scope}[red,scale=1.8]
  \foreach \y in {1/4,1/6,1/8,1/10,1/12}{% node on the grid we have drawn 
    \node[draw,circle,inner sep=0.5pt,fill] at (1-\y,2-2*\y) {};  }
  \foreach \y in {1/6,1/9,1/12,1/15,1/18
                 ,2/9,2/15,2/21}{% node on the grid we have drawn 
    \node[draw,circle,inner sep=0.5pt,fill] at (1-\y,3-3*\y) {};  }
  \foreach \y in {1/8,1/12,1/16,1/20,1/24
                  ,2/12,2/20,2/28
                  ,3/8,3/16,3/20,3/28}{ 
    \node[draw,circle,inner sep=0.5pt,fill] at (1-\y,4-4*\y) {};  }
  \foreach \y in {1/12,1/18,1/24,1/30
                 ,2/18,2/30
                 ,4/18,4/30} {
    \node[draw,circle,inner sep=0.5pt,fill] at (2-2*\y,3-3*\y) {};}
  \foreach \y in {1/12,1/18,1/24
                 ,3/12,3/24,3/30,3/42} {
    \node[draw,circle,inner sep=0.5pt,fill] at (3-3*\y,2-2*\y) {};}
  \foreach \y in {1/36,1/24 % 1/12
  				   ,2/36 % 2/12
  				   ,3/24 % 3/12,
  				   ,5/48,5/36,5/24 % 5/12
  				   %6/12,
  				   ,9/48,9/24 % 9/12} { % (12r-l)/r
    }{\node[draw,circle,inner sep=0.5pt,fill] at (3-3*\y,4-4*\y) {};}
    %\begin{scope}[shift=(0,0)]
      \foreach \y in {1/48,1/36,1/24 % 1/12
  				   ,2/36 % 2/12
  				   ,4/36 % 4/12,
  				   ,5/48,5/36,5/24 % 5/12
  				   ,8/36 % 8/12} { % (12r-l)/r
      }{\node[draw,circle,inner sep=0.5pt,fill] at (4-4*\y,3-3*\y) {};}
    %\end{scope}
  \end{scope}
  \begin{scope}[green,scale=1.8]
    \draw[domain=1.24:5.4,smooth,variable=\x] plot ({\x},{\x/(\x-1)});
    \draw[domain=1.62:5.4,smooth,variable=\x] plot ({\x},{2*\x/(\x-1)});
    \draw[domain=2.47:5.4,smooth,variable=\x] plot ({\x},{\x/(\x-2)});
    \draw[domain=2.32:5.4,smooth,variable=\x] plot ({\x},{3*\x/(\x-1)});
    \draw[domain=3.24:5.4,smooth,variable=\x] plot ({\x},{2*\x/(\x-2)});
    \draw[domain=3.71:5.4,smooth,variable=\x] plot ({\x},{\x/(\x-3)});
  \end{scope}
\end{tikzpicture}
\end{center}
\caption{Negative dilation part of  $S$ viewed in  $(\uu, \vv)$-coordinates}
\label{fig:uv-prime-coord}
\end{figure}
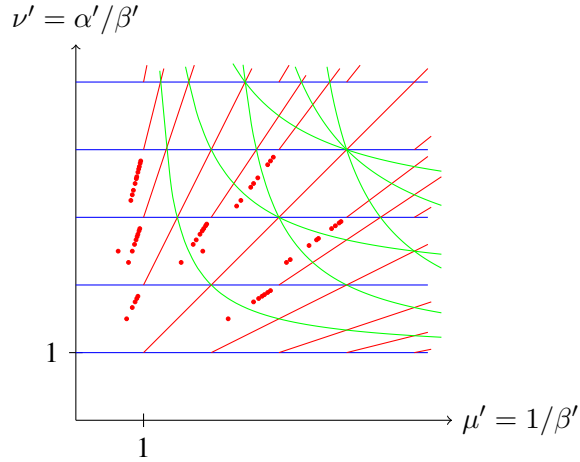
%%%%%%%%%%%%%%%%%%%%%%%%%%%%%%
%
%  END Figure N4.1 S in (\uu, \vv) coordinates 
%
%%%%%%%%%%%%%%%%%%%%%%%%%%%%%%%

In  $(\uu, \vv)$ coordinates, 
Theorem \ref{thm:symmetries-n}  says the following. 
%%%%%%%%%%%%%%%%%%%%%
% Theorem 4.1
%%%%%%%%%%%%%%%%%%%%
\begin{thm}[Linear Symmetries] 
\label{thm:symmetries-uv}
For any integer $m \ge 1$, solutions to 
${ [f_{-\vv/\uu}, f_{-1/\uu}] \geq 0 }$ are preserved under the maps
\[ (\uu, \vv) \mapsto (m\uu, \vv) \quad\text{and}\quad (\uu,\vv) \mapsto (\uu, m\vv)\]
\end{thm}

The linear  symmetries of the 
%nonnegative commutator relation 
set $S$ are   visually  apparent in Figure \ref{fig:uv-prime-coord}
after this coordinate change.

%%%%%%%%%%%%%%%%%%%%%%%%%%%%%%%%%
%
% SUBSECTION N 4.2  Lattice Disjointness Criterion
%
%%%%%%%%%%%%%%%%%%%%%%%%%%%%%%%%%
\subsection{Lattice disjointness criterion: negative dilations}
\label{sec:neg-disjointness1}

We reformulate the nonnegative  commutator relation  in terms of
the new parameters $(\uu,\vv)$ as follows.
The criterion involves a disjointness property 
of a rectangular lattice $\Lambda_{\uu,\vv}$
from an ``enlarged diagonal set'' $\cDprime$; 
see (P2') of Proposition~\ref{prop:diagonal-neg} below.

%%%%%%%%%%%%%%%%%%%%%%%%%%%%%%%%%
%
% SUBSubSECTION N 4.2.1 Approximate diagonal
%
%%%%%%%%%%%%%%%%%%%%%%%%%%%%%%%%%
\subsubsection{Enlarged diagonal set $\cDprime$}
\label{subsubsec:approx-diag1}

The positive dilation case in Part I
 used an ``approximate diagonal'' set of the plane $\RR^2$ given by 
\begin{equation}\label{eq:approx-diag}
 \mathcal{D} := \bigcup_{n \in \ZZ} \{ (x,y) : n < x < n+1  \text{ and }   n < y < n+1 \}, 
\end{equation}
see \cite[Definition 5.1]{LagR:2018a}.
The  set $\cD$ is invariant under 
the negation map $(x,y) \mapsto (-x, -y)$
and reflection map $(x,y) \mapsto (y,x)$.
It is pictured in Figure \ref{fig:diagonal-neg}(a). 
The negative dilation case will use instead the following set. 
%%%%%%%%%
% Defn. 4.2 Enlarged
%%%%%%%%
\begin{defi}\label{def:regioncDprime}
The  {\em  punctured enlarged diagonal set} 
$\cDprime$ in $\RR^2$ is  the region
\begin{equation*}
 \cDprime := \bigcup_{n\in\ZZ} \{(x,y) :  n\leq x\leq n+1,\, n < y < n+1,\, \text{ and }\, x\neq y\}.
 %\{ (x,y) : \floor{y} \leq x \leq \ceil{y} \text{ and } x \neq y \}.
\end{equation*}
\end{defi}
Here $\cDprime$  differs from $\mathcal{D}$
in being   a ``punctured approximate diagonal'' since it omits the exact diagonal line;
it is pictured in  Figure \ref{fig:diagonal-neg}(b). 
It is characterized in terms of  rounding functions by 
\begin{align}
\label{eq:dprime-round}
 \cDprime &= \{ (x,y) : \text{either} \quad \floor{y} \leq x < y \quad\text{or}\quad y < x \leq \ceil{y} \},
\end{align}
and in terms of strict rounding functions by
 \begin{align}
 \label{eq:dprime-strict}
\cDprime &= \{(x,y) :\text{either} \quad \sfloor{x} < y < x \quad\text{or}\quad x < y< \sceil{x} \} .
\end{align}
The set $\cD \cup \cDprime$ will also be important in the arguments in Sections 5 and 6;
it is  pictured in Figure \ref{fig:diagonal-neg}(c).

%%%%%%%%%%%%%%%%%%%
%
% Figure 4.2: Regions D and D' and D \cup D'
%
%%%%%%%%%%%%%%%%%%%

\begin{figure}[h]
\minipage{0.32\textwidth}
\begin{tikzpicture}[scale=0.55]
  \draw[->] (2,1) -- (5.2,1) node[right] {$x$};
  \draw[->] (1,2) -- (1,5.2) node[above] {$y$};
  \draw[-] (-0.2,1) -- (0,1); %continues x,y-axes on negative side
  \draw[-] (1,-0.2) -- (1,0);
  \draw[scale=1] (2,1.1) -- (2,1-0.1) node[below] {1};
  \draw[scale=1] (1.1,2) -- (1-0.1,2) node[left] {1};
  
  \foreach \n in {0,1,2,3,4} {
    \fill[gray,nearly transparent] (\n,\n) -- (\n,\n+1) -- (\n+1,\n+1) -- (\n+1,\n) -- cycle;
    \draw[dashed] (\n,\n) -- (\n,\n+1);
    \draw[dashed] (\n,\n+1) -- (\n+1,\n+1);
    \draw[dashed] (\n+1,\n+1) -- (\n+1,\n);
    \draw[dashed] (\n+1,\n) -- (\n,\n);
  }
\end{tikzpicture}
\\
\centering
(a) $\cD$ 
\endminipage\hfill
\minipage{0.32\textwidth}
\begin{tikzpicture}[scale=0.55]
  \draw[->] (2,1) -- (5.2,1) node[right] {$x$};
  \draw[-] (-0.2,1) -- (0,1);
  \draw[->] (1,-0.2) -- (1,5.2) node[above] {$y$};
  \draw[scale=1] (2,1.1) -- (2,1-0.1) node[below] {1};
  \draw[scale=1] (1.1,2) -- (1-0.1,2) node[left] {1};
  
  \foreach \n in {0,1,2,3,4} {
    \fill[gray,nearly transparent] (\n,\n+0.05) -- (\n,\n+1) -- (\n+0.95,\n+1) -- cycle;
    \fill[gray,nearly transparent] (\n+0.05,\n) -- (\n+1,\n) -- (\n+1,\n+0.95) -- cycle;
    \draw[-] (\n,\n) -- (\n,\n+1);
    \draw[dashed] (\n,\n+1) -- (\n+1,\n+1);
    \draw[-] (\n+1,\n+1) -- (\n+1,\n);
    \draw[dashed] (\n+1,\n) -- (\n,\n);
    \draw[dashed] (\n,\n) -- (\n+1,\n+1);
  }
%  \fill[gray,nearly transparent] (5,5) -- (5,5.2) -- (5.2,5.2) -- (5.2,5) -- cycle; 
\end{tikzpicture}
\\ \centering
(b) $\cDprime$
\endminipage\hfill
\minipage{0.32\textwidth}%
\begin{tikzpicture}[scale=0.55]
  \draw[->] (2,1) -- (5.2,1) node[right] {$x$};
  \draw[-] (-0.2,1) -- (0,1);
  \draw[->] (1,-0.2) -- (1,5.2) node[above] {$y$};
  \draw[scale=1] (2,1.1) -- (2,1-0.1) node[below] {1};
  \draw[scale=1] (1.1,2) -- (1-0.1,2) node[left] {1};
  
  \foreach \n in {0,1,2,3,4} {
    \fill[gray,nearly transparent] (\n,\n+0.01) -- (\n,\n+0.99) -- (\n+1,\n+0.99) -- (\n+1,\n+0.01) -- cycle;
    %\fill[gray,nearly transparent] (\n+0.05,\n) -- (\n+1,\n) -- (\n+1,\n+0.95) -- cycle;
    \draw[-] (\n,\n) -- (\n,\n+1);
    \draw[dashed] (\n,\n+1) -- (\n+1,\n+1);
    \draw[-] (\n+1,\n+1) -- (\n+1,\n);
    \draw[dashed] (\n+1,\n) -- (\n,\n);
    %\draw[dashed] (\n,\n) -- (\n+1,\n+1);
  }
%  \fill[gray,nearly transparent] (5,5) -- (5,5.2) -- (5.2,5.2) -- (5.2,5) -- cycle; 
\end{tikzpicture}
\\ \centering 
(c) $\cD \cup \cDprime$
\endminipage
\caption{Regions $\cD$, % = \{\floor{y}<x<\ceil{y}\}$ 
 $\cDprime$ and $\cD \cup \cDprime$ in $\RR^2$, in gray.}
\label{fig:diagonal-neg}
\end{figure}

%\newpage
%%%%%%%%%%%%%%
% Remark 4.3
%%%%%%%%%%%%%
\begin{rmk}
%We remark on symmetries of the set $\cDprime$.
%The set 
$\cDprime$ is invariant under the negation map  $(x,y) \mapsto (-x, -y)$. 
However, it is not  invariant under the reflection map $(x,y) \mapsto (y,x)$.
%the problem concerns boundary points. 
%Under reflection, the only parts of $\cDprime$  not  sent to $\cDprime$ are
%boundary points where $x$
%is an integer, since 
%horizontal edges with $y \in \ZZ$ are excluded from $\cDprime$ 
%while 
%vertical edges with $x \in \ZZ$ are included in $\cDprime$,  
%aside from integer points $(x,y) \in \ZZ^2$. 
%These edges are interchanged under the reflection, so the reflection 
% is anti-invariant on the boundary where either $x$ or $y$ is an integer.
%This lack of reflection symmetry 
%can be seen from comparing \eqref{eq:dprime-round} and \eqref{eq:dprime-strict}. 
%because
%  $\floor{y} \leq x$ is equivalent to 
% $ y< \floor{x}+1= \sceil{x} $ (a strict inequality) and analogously 
% $x\leq \ceil{y}$ is equivalent to $\sfloor{x} < y$.
\end{rmk}
%For $\cDprime$ the lack of perfect reflection symmetry $(x,y) \mapsto (y,x)$
%for $\cDprime$ is an important way  
%the negative dilation case differs from the positive dilation case.
%and it breaks the birational symmetry satisfied
%in the positive dilation case. 

%%%%%%%%%%%%%%%%%%%
% Proposition  4.4 (Sec 4.1)
%%%%%%%%%%%%%%%%%%%
\begin{prop}[Lattice disjointness criterion: negative dilations]
\label{prop:diagonal-neg}
For $\uu,\vv> 0$, the following three properties are equivalent. 
\begin{enumerate}
\item[\textnormal{(P1')}] The nonnegative commutator relation holds:
\[
[f_{-\vv/\uu}, f_{-1/\uu}](x) \geq 0 \quad \mbox{for all }  x \in \RR.
\]

\item[\textnormal{(P2')}] The two-dimensional rectangular lattice 
$ \Lambda_{\uu,\vv} = \uu \ZZ \times \vv\ZZ$ 
is disjoint from the region $\cDprime= \{ (x,y) : \floor{y} \leq x \leq \ceil{y},\, x\neq y \}$. 
That is,
\[
\Lambda_{\uu,\vv} \bigcap \cDprime = \emptyset.
\]

\item[\textnormal{(P3')}]
The set $\Lambda^+_{\uu,\vv}$ is disjoint from $\cDprime$,
where  $\Lambda_{\uu,\vv}^+ = \bigcup_{k\in \ZZ}\left( \Lambda_{\uu,\vv} + (k,k) \right) $
is the union of all 
translates of $\Lambda_{\uu,\vv}$ by  integer diagonal vectors $(k,k)$.
%Equivalently 
%\[ \Lambda_{\uu,\vv}^+ = \rowspan\begin{pmatrix}
%\uu & 0 \\ 
%0 & \vv \\
%1 & 1
%\end{pmatrix}. \]
\end{enumerate}
\end{prop}

%%%%%%%%%%%%%%%%%%%
% Proof of Proposition 4.4
%%%%%%%%%%%%%%%%%%%
\begin{proof}
(P1') $\Leftrightarrow$ (P2')
The argument parallels the proof in \cite{LagR:2018a}
of (P1) $\Leftrightarrow$ (P2) in Proposition 5.2.
We omit the details.
The main step is that
\begin{equation}\label{eq:rounding-neg2}
\sfloor{n}_{\vv/\uu} > \sfloor{n}_{1/\uu} \quad\text{for some }n\in\ZZ ,
\end{equation}
from Proposition \ref{prop:rounding-neg},
is equivalent to the condition
\begin{equation}\label{eq:disjoint-neg1}
\text{there exist } m, n \in \ZZ \text{ such that }  \sfloor{n\uu} < m\vv < {n\uu}.
\end{equation}

(P2') $\Leftrightarrow$ (P3')
Since $\Lambda_{\uu,\vv} \subseteq \Lambda^+_{\uu,\vv}$,
the implication (P3') $\Rightarrow$ (P2') holds.
In the other direction, suppose $\Lambda_{\uu,\vv}$ is disjoint from $\cDprime$.
The enlarged diagonal set $\cDprime$ is sent to itself by translation 
by an integer diagonal vector $(k,k)$,
so the translated lattice $\Lambda_{\uu,\vv} + (k,k)$ is disjoint from $\cDprime$.
This holds for all integers $k$, so
$\Lambda_{\uu,\vv}^+$ is disjoint from $\cDprime$.
Consequently (P2') $\Rightarrow$ (P3').
\end{proof}

%%%%%%%%%%%%%%%%%%%%%%%%%%%%%%%%%%%%%%%%%%%%%%%%
%%
%%  4 3 Extension of GCD-LCM
%%
%%%%%%%%%%%%%%%%%%%%%%%%%%%%%%%%%%%%%%%%%%%%%%%%

\subsection{Extension of GCD and LCM to positive real numbers}
\label{subsubsec:rational-gcd}

We  formulate an extension of the notion of  $\gcd$ and $\lcm$ to pairs of 
positive real numbers, for use in arguments to prove the linear fractional 
rescaling symmetries.
This extension is based on  commensurability of lattices. 
 
%%%%%%%%%%%%%%%%%%%
% Definition 4.5 Extended GCD
%%%%%%%%%%%%%%%%%%%
\begin{defi}
Given real numbers $u, v>0$, we define their   {\em extended greatest common divisor} by
\[ \egcd(u,v)  := \begin{cases}
w &\text{if } u\ZZ + v\ZZ = w\ZZ, \, w > 0  \\
0 & \text{if } u\ZZ + v\ZZ \text{ is dense in }\RR.
\end{cases}\]
We also define their {\em extended least common multiple}
\[ \elcm(u,v)  := \begin{cases}
w &\text{if } u\ZZ \cap v\ZZ = w\ZZ, \, w > 0  \\
+\infty & \text{if } u\ZZ \cap v\ZZ  = \{0\}.
\end{cases}\]
\end{defi}
These definitions match the usual notions when $u,v$ are  positive integers. 
We will frequently use the  following proposition, whose
proof we omit.

%%%%%%%%%%%%%%%%%%%%%
% Proposition 4.6
%%%%%%%%%%%%%%%%%%%
\begin{prop}
Suppose $u,v$ are positive real numbers.
\begin{enumerate}[(i)]
\item We have $d = \egcd(u,v) > 0$ if and only if there are coprime positive integers $p,q$ such that
$u = pd \quad\text{and}\quad v = qd.$

\item 
We have $m = \elcm(u,v) < +\infty$ if and only if there are coprime positive integers $p,q$ such that
$ m = qu \quad \text{and}\quad m = pv.$
%\item
%We have  $\egcd(u, v) >0$ 
%if and only if $\elcm(u,v) < +\infty$;
%in this case we have
%\begin{equation*}
%uv = \egcd(u,v) \elcm(u,v).
%%\qquad \text{provided}\gcd(u,v)>0.
%\end{equation*}

\end{enumerate}
\end{prop}

%%%%%%%%%%%%%%%%%%%%%%%%%%%%%%%%%%%
%
% Subsection 4.4   Disjointness of Lattice and Regions Criteria: Sporadic Cases
%
%%%%%%%%%%%%%%%%%%%%%%%%%%%%%%%%%%%%%%%%
\subsection{Disjointness of lattices and regions: part 1}
\label{sec:disjoint}

We prove several lemmas characterizing when various lattices
are disjoint from $\cDprime$, resp. $\cD \cup \cDprime$, 
to facilitate use of Proposition \ref{prop:diagonal-neg}.
The disjointness criteria involving  $\cD \cup \cDprime$  are
needed for the necessity proofs in Section~\ref{sec:negative-N}.

The first two lemmas will be  used
to show  existence of the type $(iii^{\ast})$ 
%sporadic
 rational solutions in  the sufficiency proof in Section \ref{sec:negative-S}.

%%%%%%%%%%%%%%%%%%%
% Lemma 4.7 [formerly 4.8]
%%%%%%%%%%%%%%%%%%%
\begin{lem}[Diagonal expansion-contraction]
\label{lem:diag-expansion} Let $r \ge 2$ be an integer, and
suppose $\muu , \nuu $ are  real parameters satisfying $1+\muu > \nuu$.
Then the following conditions are equivalent.
\begin{enumerate}
\item[\textnormal{(S1')}] 
The lattice 
$\Lambda = \rowspan_\ZZ\begin{pmatrix}
1+\muu & \nuu \\
1 & 1 \end{pmatrix}$
is disjoint from $\cD\cup \cDprime$.

\item[\textnormal{(S2')}] 
The lattice 
$\Lambda = \rowspan_\ZZ\begin{pmatrix}
1+\muu & \nuu \\
1 & 1 \end{pmatrix}$
is disjoint from $ \cDprime$.

\item[\textnormal{(S3')}] 
 The lattice
$\Lambda^{(r)} = \rowspan_\ZZ\begin{pmatrix}
1+\frac{1}{r}\muu & \frac1{r}\nuu \phantom{\Big|}\\
\frac1{r} & \frac1{r} \phantom{\Big|} \end{pmatrix} $
is disjoint from $\cDprime$.

\end{enumerate}
\end{lem}

\begin{proof}
To see that (S1') $\Leftrightarrow$ (S2'), observe that the lattice $\Lambda$ 
only intersects the  diagonal $\Delta$ of $\RR^2$
at integer points due to the assumption that $1+u >v$.
Thus $\Lambda$ must be disjoint from 
$\cD \smallsetminus \cDprime = \Delta \smallsetminus \ZZ^2$.

Next we show that (S2') $\Leftrightarrow$ (S3').
Observe that $\Lambda$ is disjoint from $ \cDprime$ if and only if 
the point $(1+\muu, \nuu)$ lies below (or to the right of) 
the diagonal region $\cDprime$, 
which holds if and only if $\nuu \leq  \ceil{\muu}$.  
See the left side of Figure \ref{fig:diag-expansion}.

Since the lattice $\Lambda^{(r)}$ has $(\frac{1}{r}, \frac{1}{r})$ as a generator, 
$\Lambda^{(r)}$ is invariant under translation by $\frac{j}{r}(1,1)$ for any integer $j$.
Hence in order for $\Lambda^{(r)}$ to be disjoint from $\cDprime$, 
it must also be disjoint with each translate $\frac{j}{r}(1,1) + \cDprime$, for integers $j=1,\ldots, r-1$.
The latter fact holds if and only if the other generator $(1+\frac{1}{r}\muu , \frac1{r}\nuu )$ 
of $\Lambda^{(r)}$ lies below the union of all translates $\frac{j}{r}(1,1) + \cDprime$,
which is equivalent to $\frac1{r}\nuu \leq \ceil{ \frac1{r}\muu }_{1/r}$.
See the right side of Figure \ref{fig:diag-expansion}.

Finally, we observe that for any fixed integer $r \ge 1$, the condition 
on $(\muu, \nuu)$ that
$\frac{1}{r}\nuu \leq \ceil{ \frac1{r}\muu }_{1/r}$ 
is equivalent to $\nuu \leq \ceil{\muu}$. 
This equivalence implies the equivalence of the conditions (S2') and (S3').
\end{proof}

%%%%%%%%%%%%%%%%%%%
%
% Figure 4.3 Lattices $\Lambda$ and $\Lambda^{(r)}$
%
%%%%%%%%%%%%%%%%%%%
\begin{figure}[h]
\begin{center}
\begin{tikzpicture}[scale=0.8]
  %\draw[-] (-1,0) -- (0,0);
  \draw[->] (1,0) -- (4.2,0) node[right] {$x$};
  \draw[->] (0,-1) -- (0,4.2) node[above] {$y$};
  \draw[scale=1] (1,0.1) -- (1,-0.1) node[below] {1};
  \draw[scale=1] (0.1,1) -- (-0.1,1) node[left] {1};
  
  %corner box
  \foreach \n in {-1,0,1,2,3} {
    \fill[gray,nearly transparent] (\n,\n+0.05) -- (\n,\n+1) -- (\n+0.95,\n+1) -- cycle;
    \fill[gray,nearly transparent] (\n+0.05,\n) -- (\n+1,\n) -- (\n+1,\n+0.95) -- cycle;
    %\fill[gray,nearly transparent] (\n,\n+0.05) -- (\n,\n+1) -- (\n+0.95,\n+1) -- cycle;
    %\fill[gray,nearly transparent] (\n+0.05,\n) -- (\n+1,\n) -- (\n+1,\n+0.95) -- cycle;
    \draw[-] (\n,\n) -- (\n,\n+1);
    \draw[-] (\n+1,\n+1) -- (\n+1,\n);
    \draw[dashed] (\n,\n+1) -- (\n+1,\n+1);
    \draw[dashed] (\n+1,\n) -- (\n,\n);
    \draw[dashed] (\n,\n) -- (\n+1,\n+1);
  }
  \foreach \n in {-3,-2,-1,0,1,3,4,5} {
    \node[draw,circle,inner sep=0.9pt,fill,red] at (-\n/5,\n/2) {};  }
  \foreach \n in {-1,0} {
    \node[draw,circle,inner sep=0.9pt,fill,red] at (-1-\n/5,-1+\n/2) {};  }
  \foreach \n in {-5,-4,-3,-2,-1,0,1,2,3,4,5,6} {
    \node[draw,circle,inner sep=0.9pt,fill,red] at (1-\n/5,1+\n/2) {};  }
  \foreach \n in {-7,-6,-5,-4,-2,-1,0,1,2,3,4,5} {
    \node[draw,circle,inner sep=0.9pt,fill,red] at (2-\n/5,2+\n/2) {};  }
  \foreach \n in {-5,-4,-3,-2,-1,0,1,2,3} {
    \node[draw,circle,inner sep=0.9pt,fill,red] at (3-\n/5,3+\n/2) {};  }
  \foreach \n in {-1,0,1} {
    \node[draw,circle,inner sep=0.9pt,fill,red] at (4-\n/5,4+\n/2) {};  }
  %\draw[-] (1+0.2,1) node[above left] {$(1,1)$};
  \draw[-] (6/5,1/2) node[right] {$(1+\muu,\nuu)$};
  \draw[-,red] (6/5-0.1,1/2-0.1) -- (1/5+0.1,-1/2+0.1);
  
\end{tikzpicture}
%%%%%%%
\qquad \qquad 
%%%%%%%
\begin{tikzpicture}[scale=0.8]
  \draw[-] (-1,0) -- (0,0);
  \draw[->] (1,0) -- (4.2,0) node[right] {$x$};
  \draw[->] (0,-1) -- (0,4.2) node[above] {$y$};
  \draw[scale=1] (1,0.1) -- (1,-0.1) node[below] {1};
  \draw[scale=1] (0.1,1) -- (-0.1,1) node[left] {1};
  %\draw[scale=1] (3,0.1) -- (3,-0.1) node[below] {3};
  %\draw[scale=1] (0.1,3) -- (-0.1,3) node[left] {3};
  
  % corner box
  \foreach \n in {-1,0,1,2,3} {
    \fill[gray,nearly transparent] (\n,\n+0.05) -- (\n,\n+1) -- (\n+0.95,\n+1) -- cycle;
    \fill[gray,nearly transparent] (\n+0.05,\n) -- (\n+1,\n) -- (\n+1,\n+0.95) -- cycle;
    \draw[-] (\n,\n) -- (\n,\n+1);
    \draw[-] (\n+1,\n+1) -- (\n+1,\n);
    \draw[dashed] (\n,\n+1) -- (\n+1,\n+1);
    \draw[dashed] (\n+1,\n) -- (\n,\n);
    \draw[dashed] (\n,\n) -- (\n+1,\n+1);
  }
  %hashlines inside corner rectangle
  %\foreach \n in {1,2} {
  %  \draw[dotted] (\n/3,0) -- (\n/3,1);
  %  \draw[dotted] (0,\n/3) -- (1,\n/3);
  %  
  %  \draw[dotted] (\n/3,\n/3) -- (\n/3-1,\n/3) -- (\n/3-1,\n/3-1) -- (\n/3,\n/3-1) -- cycle;
  %}
  %points of cyclic subgroup
  \foreach \n in {-2,-1,0,1,2,3,4,5,6,7,8,9,10,11} {
    \node[draw,circle,inner sep=0.4pt,fill,red] at (\n/3,\n/3) {};  }
  \foreach \n in {-3,-2,-1,0,1,2,3,4,5,6,7,8,9} {
    \node[draw,circle,inner sep=0.4pt,fill,red] at (-1/15+\n/3,1/6+2/3+\n/3) {};  }
  \foreach \n in {-2,-1,0,1,2,3,4,5,6,7} {
    \node[draw,circle,inner sep=0.4pt,fill,red] at (-2/15+\n/3,2/6+4/3+\n/3) {};  }
  \foreach \n in {-1,0,1,2,3,4,5} {
    \node[draw,circle,inner sep=0.4pt,fill,red] at (-3/15+\n/3,3/6+6/3+\n/3) {};  }
  \foreach \n in {-1,0,1,2} {
    \node[draw,circle,inner sep=0.4pt,fill,red] at (-4/15+\n/3,20/6+\n/3) {};  }
  \foreach \n in {-1} {
    \node[draw,circle,inner sep=0.4pt,fill,red] at (-5/15+\n/3,25/6+\n/3) {};  }
  \foreach \n in {0,1,2,3,4,5,6,7,8,9,10,11,12} {
    \node[draw,circle,inner sep=0.4pt,fill,red] at (1/15+\n/3,-1/6-2/3+\n/3) {};  }
  \foreach \n in {2,3,4,5,7,8,9,10,11} {
    \node[draw,circle,inner sep=0.4pt,fill,red] at (2/15+\n/3,-10/6+\n/3) {};  }
  \foreach \n in {5,6,7,9,10,11} {
    \node[draw,circle,inner sep=0.4pt,fill,red] at (3/15+\n/3,-15/6+\n/3) {};  }
  \foreach \n in {7,8,9,10,11} {
    \node[draw,circle,inner sep=0.4pt,fill,red] at (4/15+\n/3,-20/6+\n/3) {};  }
  \foreach \n in {10,11} {
    \node[draw,circle,inner sep=0.4pt,fill,red] at (5/15+\n/3,-25/6+\n/3) {};  }
    
  %label (1,1) point
  %\draw[-] (1-0.2,1+0.1) node[above] {$(1,1)$};
  \node[draw,circle,inner sep=0.9pt,fill,red] at (1,1) {};
  
  %"fundamental" corner-cutting segment
  \draw[-] (1+1/15,1/6+0.1) node[right] {$(1+\frac1r\muu,\frac1r\nuu)$};
  \node[draw,circle,inner sep=0.9pt,fill,red] at (1/15+1,-5/6+1) {};
  \draw[-,red] (16/15-0.1,1/6-0.1) -- (11/15+0.1,-1/6+0.1);
  
\end{tikzpicture}
\end{center}
\caption{Lattices $\Lambda$ and   $\Lambda^{(r)}$ with $r = 3$ in red, 
shown with punctured approximate diagonal region $\cDprime$ in gray.}
\label{fig:diag-expansion}
\end{figure}
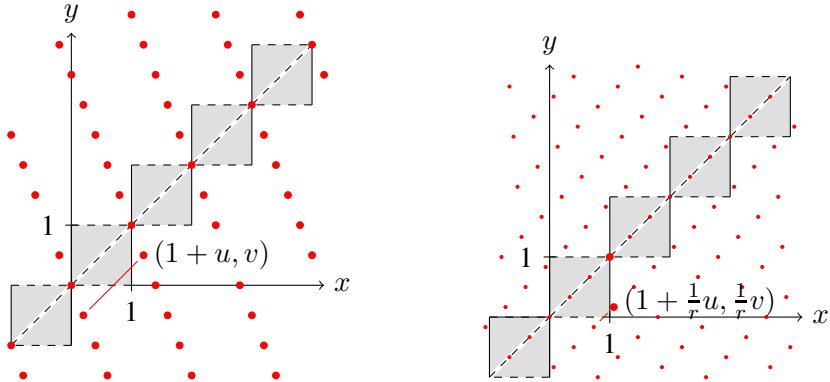
%%%%%%%%%%%%%%%%%%%
%
% END Figure 4.3 Lattices $\Lambda$ and $\Lambda^{(r)}$
%
%%%%%%%%%%%%%%%%%%%

 The condition (a) of the next lemma provides a recipe for producing an infinite number of new rational solutions
 in $S$,
 starting from an initial rational  solution in $S$ associated to  $\Lambda_{\ell p,\ell q}$.
%%%%%%%%%%%%%%%%%%%
% Lemma 4.8  [formerl´ 4.9]
%%%%%%%%%%%%%%%%%%%
\begin{lem}
\label{lem:hyperbola-neg-sporadic}
Suppose $p,q\geq 1$ are coprime integers and 
$\ell>0$ is a real number  such that 
the lattice $\Lambda_{\ell p,\ell q}$ is disjoint from the set
\[  \cDprime = \bigcup_{n\in\ZZ} 
    \{ (x,y) : n\leq x \leq n+1, \, n < y < n+1, \,\text{ and }\, x\neq y\}. \]
\begin{enumerate}[(a)]
\item[(a)]
If $\ell > 0$ is rational,
then for any integer $r \geq 1$, the lattice $\Lambda_{\lambda p,\lambda q}$ with
\[ 
\lambda = \lambda_{r,\ell} = 1 + \frac1{r}\left(\ell - 1\right) 
\]
is disjoint from $\cDprime$.
(Note that when $r = 1$, $\lambda_{1,\ell} = \ell$.)

\item[(b)]
If $\ell$ is irrational,
then necessarily $\ell > 1$.
Conversely, if $\ell > 1$ is irrational then
the lattice $\Lambda_{\ell p,\ell q}$ is  disjoint from the set
$\cDprime$.

\end{enumerate}
\end{lem}
\begin{proof}
Set $\uu = \ell p$, $\vv = \ell q$. 
Note that $\ell = \egcd(\uu,\vv) $.
Recall that $\Lambda_{\ell p, \ell q} = \Lambda_{\uu,\vv}$ denotes 
the image of 
$(m,n) \mapsto (m\uu, n\vv) $
for $(m,n) \in \ZZ^2$.
Equivalently, $\Lambda_{\uu,\vv}$ is the integer row span 
of the diagonal matrix
\begin{align*}
\Lambda_{\uu,\vv} = \rowspan_\ZZ\begin{pmatrix}
\uu & 0 \\
0 & \vv
\end{pmatrix}.
%\ZZ^2 &\to \RR^2 \\
%(m,n) &\mapsto (m\uu, n\vv) .
\end{align*}
We first find a second basis for $\Lambda_{\uu,\vv}$ which is convenient
for comparison with $\cDprime$.
Namely, we choose one basis vector to lie on the main diagonal.
Let $\Delta = \{(x,x) : x\in \RR\}$ denote the  diagonal of $\RR^2$.

Let $m_0, n_0 \geq 1$ be integers such that $m_0 p - n_0 q = 1$.\medskip

{\bf Claim.}
{\em  The lattice $\Lambda_{\uu,\vv}$ is equal to
$\rowspan_\ZZ\begin{pmatrix}
m_0 \uu & n_0 \vv \\
q \uu & p \vv
\end{pmatrix} .$
}

{\em Proof}. Note that
\[ 
\begin{pmatrix}
m_0 \uu & n_0 \vv \\
q \uu & p \vv
\end{pmatrix} = 
\begin{pmatrix}
m_0  & n_0  \\
q  & p 
\end{pmatrix}\begin{pmatrix}
 \uu & 0  \\
 0 &  \vv
\end{pmatrix}.
\]
The integer change of basis matrix
$ \begin{pmatrix}
m_0  & n_0  \\
q  & p 
\end{pmatrix}$
has determinant $1$.
This proves the claim.

Now $\Lambda_{\uu,\vv}$ is disjoint from $\cDprime$ by hypothesis.
Hence
the implication (P2') $\Leftrightarrow$ (P3') in Proposition \ref{prop:diagonal-neg}
implies that $\Lambda_{\uu,\vv}^+$ is disjoint from $\cDprime$
where
\begin{align*}
\Lambda_{\uu,\vv}^+ = \rowspan_\ZZ\begin{pmatrix}
\uu & 0 \\
0 & \vv \\
1 & 1
\end{pmatrix}.
\end{align*}
The above claim implies that
 $\Lambda^+_{\uu,\vv}$ is equal to the $\ZZ$-row span of  vectors %in $\RR^2$,
\[ \Lambda^+_{\uu,\vv} = \rowspan_\ZZ \begin{pmatrix}
m_0\uu & n_0\vv \\
q\uu & p\vv \\
1 & 1 \end{pmatrix}= \rowspan_\ZZ \begin{pmatrix}
m_0\uu & n_0\vv \\
\ell pq & \ell pq \\
1 & 1 \end{pmatrix} .\]

(a)
Suppose first that $\ell$ is rational.
Since the second and third rows are on the diagonal $\Delta$ while the first row is off the diagonal, 
we have
\[ \Lambda^+_{\uu,\vv} = \rowspan_\ZZ \begin{pmatrix}
m_0\uu & n_0\vv \\
\egcd(1, \ell pq) & \egcd(1,\ell pq) \end{pmatrix}\]
by definition of the extended greatest common divisor,
and the assumption that $\ell$ is rational. % so $\egcd(1,\ell pq)>0$.
Let $s$ be the positive integer such that $\frac{1}{s} = \egcd({1},{\ell pq})$.

Next we apply Lemma \ref{lem:diag-expansion}
with  $\muu = s(m_0\uu -1)$ and $\nuu = sn_0\vv$,
so
\[ 
\Lambda^{(s)}
:= \rowspan_\ZZ\begin{pmatrix}
1+ (m_0 \uu-1) & n_0\vv \\
 \frac1{s}  & \frac1{s} \end{pmatrix}
 = \Lambda^+_{\uu,\vv} .
\]
Since $\Lambda^{(s)}$ is disjoint from $\cDprime$,
the implication (S3') $\Rightarrow$ (S2')
of Lemma \ref{lem:diag-expansion}
guarantees that
\[ 
\Lambda := \rowspan_\ZZ\begin{pmatrix}
1+ s(m_0 \uu-1) & sn_0\vv \\
 1  & 1 \end{pmatrix}
\]
is disjoint from $\cDprime$.
But then (S2') $\Rightarrow$ (S3') implies that
\[ 
\Lambda^{(rs)} := \rowspan_\ZZ\begin{pmatrix}
1-\frac{1}{r}+ \frac{1}{r}m_0 \uu & \frac{ 1}{r}n_0\vv \\
 \frac{1}{rs}  & \frac{1}{rs} \end{pmatrix}.
 \]
 is disjoint from $\cDprime$
 for any integer $r\geq 1$.

Finally, to see that the lattice $\Lambda_{\lambda p, \lambda q}$ is disjoint from $\cDprime$
it suffices to show that 
$\Lambda_{\lambda p, \lambda q}$
is contained in 
$\Lambda^{(rs)}$,
i.e.  it suffices to verify  that
\begin{align*} 
\Lambda^{(rs)} \supset \Lambda_{\lambda p, \lambda q} :=& \rowspan_\ZZ\begin{pmatrix}
\lambda p & 0  \\
0 &  \lambda q 
\end{pmatrix} 
\end{align*}
Recall that $\lambda = 1 + \frac1{r}(\ell-1)$;
this containment follows from the matrix relation
\begin{equation}
\label{eq:58}
 \begin{pmatrix}
\lambda  p & 0 \\
0 &  \lambda q
\end{pmatrix}
=  \begin{pmatrix}
p & - n_0 s \ell pq \\
-q & (r-1)sq + m_0 s \ell pq   
\end{pmatrix} 
\begin{pmatrix}
1- \frac{1}{r}+ \frac{1}{r}m_0 \ell p & \frac{1}{r}n_0 \ell q \\
 \frac{1}{rs} & \frac{1}{rs} \end{pmatrix} . 
\end{equation}
Here we recall that $m_0 p - n_0 q = 1$, and that $s \ell pq$ is an integer 
by the choice $\frac1s = \egcd(1, \ell pq)$.
Thus the second matrix in  \eqref{eq:58} has integer entries.
This shows that $\Lambda_{\lambda p, \lambda q} \subset \Lambda^{(rs)}$
so $\Lambda_{\lambda p, \lambda q}$ is disjoint from $\cDprime$ as desired.

(b)
Now suppose $\ell$ is irrational.
As above, we consider when $\cDprime$ is disjoint from
\[ 
\Lambda^+_{\uu,\vv} = 
\rowspan_\ZZ \begin{pmatrix}
\ell m_0p & \ell n_0q \\
\ell pq & \ell pq \\
1 & 1 \end{pmatrix} .\]
%where $m_0, n_0 \geq 1$ are integers such that $m_0 p - n_0 q = 1$.
When $\ell$ is irrational, the set $\ell pq \ZZ + \ZZ$ is dense in $\RR$.
Hence the topological closure of $\Lambda^+_{\ell p, \ell q}$ inside $\RR^2$ 
contains the diagonal line $\Delta$ of $\RR^2$.
Using this observation it is straightforward to verify that 
the closure of $\Lambda^+_{\ell p, \ell q}$ is
\begin{align*} 
\overline{\Lambda^+_{\ell p, \ell q}} 
  &= \{(x,y) \in \RR^2 : x - y \in \ell \ZZ\} .
\end{align*}
If $0< \ell < 1$, then the closure $\overline{\Lambda^+_{\ell p, \ell q}}$ intersects the interior of $\cDprime$
(at a point with  $x-y = \pm \ell$)
so ${\Lambda^+_{\ell p, \ell q}}$ must also intersect $\cDprime$.
On the other hand if $\ell > 1$, then $\overline{\Lambda^+_{\ell p, \ell q}}$ is disjoint from $\cDprime$ so 
 ${\Lambda^+_{\ell p, \ell q}}$ is also disjoint from  $\cDprime$.
\end{proof}

%%%%%%%%%%%%%%%%%%%%%%%%%%%%%%%%%%%
%
% Subsection 4.5   Disjointness of Lattice and Regions Criteria: Sporadic Cases
%
%%%%%%%%%%%%%%%%%%%%%%%%%%%%%%%%%%%%%%%%
\subsection{Disjointness of lattices and regions: part 2}
\label{sec:disjoint2}

 We prove two  further lemmas  needed  for the necessity proof in Section \ref{sec:neg-necess-proof}, 
 which are not needed in the sufficiency proof in Section \ref{sec:negative-S}.
 The first  shows that every rational lattice $\Lambda_{\lambda p, \lambda q}$
 which is associated to a solution in $S$ 
 is associated  to some rational lattice $\Lambda_{\ell p, \ell q}$
 associated to a case $(i)$ solution in $S$ which is disjoint from $\cD \cup \cDprime$.
  Lemma \ref{lem:diagonal-vert-boundary} characterizes all such lattices.

%%%%%%%%%%%%%%%%%%%
% Lemma 4.9(formerl 6.6) 
%%%%%%%%%%%%%%%%%%%

\begin{lem}
\label{lem:neg-sporadic-converse}
Suppose that $p,q$ are coprime positive integers and $\lambda > 0$ is rational such that the lattice
$\Lambda_{\lambda p,\lambda q} $
is disjoint from $\cDprime$. 
Then there necessarily exists some positive integer $r$ such that for 
\[ 
\ell = \ell_{r,\lambda} =   1 + r(\lambda - 1),
\]
the lattice $\Lambda_{\ell p, \ell q}$ is disjoint from $\cD \cup \cDprime$.
\end{lem}

\begin{proof}
Assume the same notation as in the proof of Lemma~\ref{lem:hyperbola-neg-sporadic}.
By the argument there, $\cDprime$ is disjoint from 
\[ \Lambda^+_{\lambda p,\lambda q} = \rowspan_\ZZ \begin{pmatrix}
\lambda m_0 p & \lambda n_0 q  \\
\egcd(1,\lambda pq) & \egcd(1, \lambda pq) 
\end{pmatrix} .\]
Let $r$ denote the positive integer which satisfies $\frac{1}{r} = \egcd(1, \lambda pq)$.
(Recall that $\lambda$ is rational.)

Now we apply Lemma \ref{lem:diag-expansion}:
let $\Lambda^{(r)} = \Lambda^+_{\lambda p,\lambda q}$ 
so that $1 + \frac{1}{r}\muu = m_0\lambda p$ and $\frac{1}{r}\nuu = n_0\lambda q$
and 
(S3') holds,
i.e. $\Lambda^{(r)}$ is disjoint from $\cDprime$.
Then (S1') holds, namely
$$
 \Lambda := \rowspan_\ZZ\begin{pmatrix}
1 - r + \lambda m_0 pr & \lambda n_0  qr \\
1 & 1
\end{pmatrix} 
$$
is disjoint from $\cD \cup \cDprime$.
It remains to check that 
$\Lambda \supset \Lambda_{\ell p, \ell q}$
where $\ell = 1 + r(\lambda - 1)$, 
i.e. that
\[ \Lambda \supset \rowspan_\ZZ\begin{pmatrix}
(1 - r+ \lambda r)p & 0  \\
0 & (1 - r + \lambda r)q 
\end{pmatrix} .\]
This amounts to the calculation
\begin{align*}
%%%%%%%%%%%%%%%%% 
%\begin{pmatrix}
%\ell p & 0 \\
%0 & \ell q 
%\end{pmatrix} &= 
%%%%%%%%%%%%%%%%%%
\begin{pmatrix}
(1 - r+ \lambda r)p & 0  \\
0 & (1 - r + \lambda r)q 
\end{pmatrix}  
&= \begin{pmatrix}
p & - \lambda n_0 pqr \\
-q & q - rq + \lambda m_0 pqr
\end{pmatrix}
\begin{pmatrix}
1 - r + m_0\lambda pr  & n_0 \lambda qr \\
1 & 1
\end{pmatrix},
\end{align*}
where we note that $\lambda pqr $ is an integer by the choice 
$\frac{1}{r} = \egcd(1, \lambda pq)$.
This proves the disjointness of $\Lambda_{\ell p, \ell q}$ and $\cD\cup \cDprime$ as desired.
\end{proof}

The next lemma 
%used in Section \ref{sec:neg-necess-proof}, 
shows that disjointness from $\cD\cup \cDprime$  requires
that $(\uu, \vv)$ must lie  on a rectangular hyperbola or on a horizontal line. 

%%%%%%%%%%%%%%%%%%%
% Lemma 4.10 [formerly 4.7]
%%%%%%%%%%%%%%%%%%%
\begin{lem}[Combined lattice/diagonal   disjointness classification]
\label{lem:diagonal-vert-boundary}
For rational parameters $\uu, \vv> 0$ the following conditions are equivalent.
\begin{enumerate}
\item[\textnormal{(M1')}] 
The lattice 
$\Lambda_{\uu,\vv} = \rowspan_\ZZ
\begin{pmatrix}
\uu & 0 \\ 
0 & \vv \end{pmatrix}$
 is disjoint from 
$\cD\cup \cDprime$.

\item[\textnormal{(M2')}] 
There exist integers $m\geq 0$, $n\geq 1$ such that
  \begin{equation*}
  \frac{m}{\uu} + \frac{n}{\vv} = 1.
  \end{equation*}
\end{enumerate}
\end{lem}

\begin{proof}
(M2')$ \Rightarrow $(M1')
Suppose there exist integers $m\geq 0, n\geq 1$ such that $\frac{m}{\uu} + \frac{n}{\vv}=1$. 
By % applying first 
\cite[Theorem 5.4]{LagR:2018a}
we have the relation $[f_{1/\uu},f_{\vv/\uu}] \geq 0$,
and then the equivalence
 \cite[Proposition 5.2]{LagR:2018a}  
implies
$\Lambda_{\uu,\vv}$ is disjoint from $\cD$. 
%By Theorem \ref{thm:neg-sufficient} we have the relation $[f_{-\vv/\uu},f_{-1/\uu}] \geq 0$,
%and the equivalence in Proposition \ref{prop:diagonal-neg} 
The arguments used in proving \cite[Theorem 5.4 and Lemma 5.5]{LagR:2018a}
also imply that $\Lambda_{\uu,\vv}$
 is disjoint from $\cDprime$. 
Thus (M1') holds.
In this direction, it is not necessary to assume that $\uu, \vv$ are rational.

(M1')$ \Rightarrow $(M2')
Now suppose that $\Lambda_{\uu,\vv}$ is disjoint from $\cD\cup\cDprime$.

{\em Case 1.} 
Suppose that $\uu$ is not an integer. 
By hypothesis $\Lambda_{\uu,\vv}$ is disjoint from $\cD$, 
so by \cite[Proposition 5.2]{LagR:2018a}
there  exist integers $m,n\geq 0$ such that
$ \frac{m}{\uu} + \frac{n}{\vv} = 1$.
Since $\uu$ is not an integer
we must have $n\geq 1$ in any such relation so (M2') holds.

{\em Case 2.}
Suppose $\uu$ is an integer. Then set
 $\lambda := \egcd(\uu, \vv)$.
Since $\uu,\vv$ are both rational, 
$\lambda$ is a positive rational number.
We claim that $\lambda \geq 1$.

By definition of extended greatest common divisor, there exist integers $m_0, n_0$ such that
${ m_0 \uu - n_0 \vv = \lambda }$.
Consider the lattice point 
$$(m_0 \uu, n_0 \vv)  = (m_0 \uu, m_0 \uu - \lambda) \in \Lambda_{\uu,\vv} .$$
Since $m_0 \uu$ is an integer, 
the region $\cDprime$ contains the open segment
$$\{ (m_0 \uu, y) : m_0 \uu -1 < y < m_0 \uu\} \subset \cDprime.$$
Thus the hypothesis that $\Lambda_{\uu,\vv}$ is disjoint from $\cDprime$
implies that $\lambda \geq 1$ as claimed.

Let $p,q$ be the coprime integers
such that $(\uu,\vv) = (\lambda p, \lambda q)$,
and let
$m = \uu - p = \lambda p - p$ and $n = q \geq 1$.
Here $m$ is an integer since $\uu$ is an integer,
and $m\geq 0$ since $\lambda \geq 1$.
Then
\[
 \frac{m}{\uu} + \frac{n}{\vv} = \frac{\lambda p-p}{\lambda p} + \frac{q}{\lambda q} 
= 1 - \frac{1}{\lambda} + \frac{1}{\lambda}  = 1 .
\]
Both $m$ and $n$ are integers, so condition (M2') holds as desired.
\end{proof}

%%%%%%%%%%%%%%%%%%%%%%%%%%%%%%%%%%%%%%%%%%%%%%%%
%%
%%  4.5 Proof of Theorem 1.5
%%
%%%%%%%%%%%%%%%%%%%%%%%%%%%%%%%%%%%%%%%%%%%%%%%%

\subsection{Proof of linear fractional symmetries: Theorem \ref{thm:lin-frac-sym}}
\label{subsubsec:integer-contact-proof}

We view  Theorem \ref{thm:lin-frac-sym} in  $(\uu, \vv)$ coordinates.
Recall that $\alpha = -\frac{\vv}{\uu}$, $\beta=-\frac{1}{\uu}$.
The   map $\Phi_{p}^{r}(\beta)$ is conjugate  under $\bJ(\uu)= -\frac{1}{\uu} := \beta$ 
to $\Psi_{p}^r(\uu)= \bJ^{-1} \circ \Phi_{p}^r \circ \bJ (\uu) $, 
which  acts linearly  in the $\uu$-coordinate. 
Theorem \ref{thm:lin-frac-sym} becomes:

%%%%%%%%%%%%%%%%%%%%%%%%%%%%%%%
%
%Theorem 4.11  (Formerly 1.5) 
%
%%%%%%%%%%%%%%%%%%%%%%%%%%%%%%
%\begin{customthm}{1.5}
\begin{thm}[Linear fractional symmetries in $(\uu, \vv)$-coordinates]
\label{thm:lin-frac-sym-uv}
If $ \frac{\vv}{\uu} = \frac{q}{p}$ is a fixed positive rational given in lowest terms, 
then for any integer $r \ge 1$ the set of $\uu$ 
satisfying $[f_{-q/p},f_{-1/\uu}] \geq 0$ 
is mapped to itself under
\[ 
\uu \mapsto \Psi_{p}^r(\uu):=  
%\frac{1}{r} \uu  - \left( \frac{1}{r} - 1\right) p = 
\frac{1}{r }\left( \uu - p \right)  + p.
\]
This map sends 
 $(\uu, \frac{q}{p} \uu) \in S $
to $( \Psi_{p}^r(\uu), \frac{q}{p}  \Psi_{p}^r(\uu)) \in S$ in $(\uu, \vv)$ coordinates.
\end{thm} 
%\end{customthm}
%%%%%%%%%%%%%%%
% Proof of Theorem 2.2.
%%%%%%%%%%%%%%
\begin{proof}
For integer $r \ge 1$,
the map  $\Psi_{p}^r(\uu)$ sends rational numbers to rational numbers and  irrationals to irrationals.
We  treat these two cases separately.

{\em Case (i).} Suppose that $\uu$ is irrational. 
Lemma \ref{lem:hyperbola-neg-sporadic} (b) says  
$[f_{-q/p},f_{-1/\uu}] \geq 0$ if and only if $\uu \geq p$.
The map $\Psi_{p}^r(\mup)$, which is equivalent to $\uu - p \mapsto \frac{1}{r} (\uu - p)$, 
sends the  set $\{ \uu : \uu \geq p \}$ to itself.

{\em Case (ii).} Suppose that $\uu$ is rational. 
Lemma \ref{lem:hyperbola-neg-sporadic}  (a)
and Proposition \ref{prop:diagonal-neg} then
imply:
if $[f_{-q/p},f_{-1/\uu}] \geq 0$ then also  ${ [f_{-q/p},f_{-1/\Psi_{p}^r(\uu)}] \geq 0 }$.
\end{proof}

%%%%%%%%%%%%%%%%%%%%%%%%%%%%%%%%%%%%%%%%%%%%%%%%%%%%%%%%%%%%%%%%%%%%%%%%%%%%%%%%%%%
%
% SECTION  5: negative dilations:sufficiency 
%
%%%%%%%%%%%%%%%%%%%%%%%%%%%%%%%%%%%%%%%%%%%%%%%%%%%%%%%%%%%%%%%%%%%%%%%%%%%%%%%%%%%
\section{Negative dilations classification: sufficiency}
\label{sec:negative-S}

In this section we prove the sufficiency part of Theorem \ref{thm:negative}.

%%%%%%%%%%%%%%%%%%%%%%%%
%
% SUBSECTION 5.1 Completion of Proof of Sufficiency in Theorem 1.3
%
%%%%%%%%%%%%%%%%%%%%%%%%%%
\subsection{Sufficiency condition in $(\uu, \vv)$-coordinates}
\label{sec:neg-sufficient}
We use the $(\uu,\vv)$ change of coordinates.
We restate the sufficiency  condition in Theorem \ref{thm:negative} in terms of $(\uu,\vv)$-coordinates.

%%%%%%%%%%%%%%%%%%%
% Theorem. 5.1
%%%%%%%%%%%%%%%%%%%
\begin{thm}[Sufficiency condition in $(\uu,\vv)$ coordinates]
\label{thm:neg-sufficient}
Given parameters $\uu,\vv>0$, 
 the nonnegative commutator relation $[f_{-\vv/\uu}, f_{-1/\uu}] \geq 0$  holds 
if any one of the following conditions holds.
\begin{enumerate}
\item[(i)] 
There are integers $m\geq 0,n\geq 1$ such that
  \begin{equation*}
  \frac{m}{\uu} + \frac{n}{\vv} = 1.
  \end{equation*}
\item[(ii)] 
  There are coprime integers $p,q\geq 1$ such that
  \begin{equation*}
  %\alpha = -\frac{n}{m}, \quad -\frac1{m} \leq \beta < 0.
  \frac{\uu}{p} = \frac{\vv}{q} \geq 1.
  \end{equation*}

\item[(iii*)]
There are coprime integers $p, q\geq 1$ and  integers $m\geq 0, n\geq 1, r\geq 1$ such that
  \begin{equation*}
  \frac{\uu}{p} = \frac{\vv}{q} = 1 +\frac{1}{r}\left(\frac{m}{p}+ \frac{n}{q} - 1 \right).
  \end{equation*}
\end{enumerate}
\end{thm}
Note that in terms of extended greatest common divisor (see Section~\ref{subsubsec:rational-gcd}),
case $(ii)$ of Theorem~\ref{thm:neg-sufficient} is equivalent to ${ \egcd(\uu,\vv) \geq 1 }$.
The condition in case $(i)$ is related to disjoint Beatty sequences, see Part I
\cite[Prop. 2.5]{LagR:2018a}.
We prove Theorem \ref{thm:neg-sufficient} in the remainder of this section.

%%%%%%%%%%%%%%%%%%%
% SUB_SECTION N5.2
%%%%%%%%%%%%%%%%%%%
\subsection{Sufficiency: Hyperbola and half-line cases}
\label{sec:neg-lattice-disjointness}

The next result  shows existence of solutions for parameters corresponding to $m=n=1$,
to $m=0$,
and to $p=q=1$. 
Later in the proof we will use the linear  symmetries of $S$ to construct solutions for other 
$m,n,p,q$ parameters.

%%%%%%%%%%%%%%%%%%%
% Lemma 5.2 Disjointness criterion satisfied
%%%%%%%%%%%%%%%%%%%
\begin{lem}[Rectangular hyperbola and half-line sufficiency]
\label{lem:hyperbola-neg}
Suppose $\uu, \vv > 0$. 
\begin{enumerate}[(i)]
\item
If  $(\uu,\vv)$ lies on the hyperbola $\frac{1}{\uu} + \frac{1}{\vv} = 1$,
 then 
 $[f_{-\vv/\uu}, f_{-1/\uu}] \geq 0$  holds.
\item 
If $(\uu, \vv)$ lies on the half-line $\vv=1$
($\uu >0$),
then $[f_{-\vv/\uu}, f_{-1/\uu}] \geq 0$  holds.
\item
If  $(\uu,\vv)$ lies on the half-line ${\uu} = {\vv} \geq 1$,
then 
$[f_{-\vv/\uu}, f_{-1/\uu}] \geq 0$  holds.
\end{enumerate}
\end{lem}

\begin{proof}
It suffices to verify for each case  that condition  (P2') of Proposition \ref{prop:diagonal-neg}
holds, whence the Proposition (P1') gives the result. 
To show that the lattice $\Lambda_{\uu,\vv} = \{ (m\uu, n\vv) : m,n \in\ZZ \}$ is disjoint from $\cDprime$
(resp. $\cD$), 
it suffices to show they are disjoint in the closed positive quadrant 
since both sets are preserved by 
$(x,y) \mapsto (-x,-y)$ and $\cDprime$ (resp. $\cD$)
does not intersect the open second or fourth quadrants.

(i) This case  parallels  Lemma 5.5 in  \cite{LagR:2018a}.

(ii) This case is clear from the definition of $\cDprime$ and $\Lambda_{\uu,\uu}$.

(iii) Suppose that $(\uu,\vv)$ lies on the half-line $\{(\uu,\vv) : \uu=\vv \geq1\}$, 
and consider the point $(m\uu, n\uu)$ in the lattice $\Lambda_{\uu,\uu}$.
Note that $\cDprime$ is contained in the open region $\{ 0 < |x - y| < 1 \}\subset \RR^2$. 
If ${ m=n }$ then the point $(m\uu, n\uu)=(m \uu, m\uu)$
is contained in the diagonal of $\RR^2$ so it is not in $\cDprime$. 
Next suppose $m\neq n$. 
Then the lattice point $(m\uu, n\uu)$ has coordinates which satisfy 
$|m\uu-n\uu|  = |m-n|\cdot |\uu| \geq |m-n| \geq 1$, 
so  again the point is not in $\cDprime$.
Thus $\Lambda_{\uu,\vv}$ is disjoint from $\cDprime$.
\end{proof}

%%%%%%%%%%%%%%%%%%%
% Subsection N 5.3 =5.1.2
%%%%%%%%%%%%%%%%%%%
\subsection{Proof of Sufficiency Theorem \ref{thm:neg-sufficient}}
\label{sec:neg-suff-proof}

%%%%%%%%%%%%%%%%%%%
% Proof of Theorem 5.1= 4.8
%%%%%%%%%%%%%%%%%%%
\begin{proof}[Proof of Theorem \ref{thm:neg-sufficient}]
There are three sufficient conditions to check.

{\em Case (i).}  Suppose  $\uu, \vv$ satisfy $\frac{m}{\uu} + \frac{n}{\vv} =1$ with $m,n\geq 1$.
Let $\uu_0 = \frac{\uu}{m}$ and $\vv_0 = \frac{\vv}{n}$;
these satisfy $\frac{1}{\uu_0} + \frac{1}{\vv_0} =1$. 
Lemma \ref{lem:hyperbola-neg} (i) implies that 
$(\uu_0,\vv_0)$ satisfies $[f_{-\vv_0/\uu_0}, f_{-1/\uu_0}] \geq 0$.
It follows from Theorem \ref{thm:symmetries-uv}
that $[f_{-\vv/\uu},f_{-1/\uu}]\geq 0$. % by Proposition \ref{prop:diagonal-neg}.

If $m = 0$, 
then let $\uu_0 = \uu$ and $\vv_0 = \frac{\vv}{n}$.
We repeat the same argument with Lemma \ref{lem:hyperbola-neg} (ii)
and Theorem \ref{thm:symmetries-uv}.

{\em Case (ii).}  Suppose  $\uu, \vv$ satisfy $\frac{\uu}{p} = \frac{\vv}{q} \geq 1$ for 
coprime integers
%\footnote{For this argument to work it is not necessary that $p,q$ be coprime, but dropping the coprime assumption does not yield additional solutions of $\uu, \vv$ parameters.}  
$p,q$, 
then  let $\uu_0 = \frac{\uu}{p}$ and $\vv_0 = \frac{\vv}{q}$. 
Now Lemma \ref{lem:hyperbola-neg} (iii) implies that 
$[f_{-\vv_0/\uu_0}, f_{-1/\uu_0}] \geq 0$,
so it follows from Theorem \ref{thm:symmetries-uv} that
$[f_{-\vv/\uu},f_{-1/\uu}]\geq 0$. 
(This argument works without assuming coprimality of $p$ and $ q$, but
dropping this assumption does not yield additional solutions.)

{\em Case (iii).}  Suppose that $\uu,\vv$ satisfy
\begin{equation}\label{eq:star} 
%\tag{$\star$}
\frac{\uu}{p} = \frac{\vv}{q} = 1 + \frac{1}{r}\left(\frac{m}{p} + \frac{n}{q} - 1\right) 
\end{equation}
for coprime integers $p,q\geq 1$ and integers $m\geq 0, n\geq 1, r\geq 1$.
Let 
\[
\frac{\uu_0}{p} = \frac{m}{p} + \frac{n}{q}
\qquad\text{and}\qquad
 \frac{\vv_0}{q} = \frac{m}{p} + \frac{n}{q} 
\]
so that 
$ \frac{m}{\uu_0} + \frac{n}{\vv_0} = 1 $.
By case $(i)$ we have
$[f_{-\vv_0/\uu_0}, f_{-1/\uu_0}] \geq 0$.

To show that $[f_{-\vv/\uu},f_{-1/\uu}]\geq 0$,
we apply Theorem \ref{thm:lin-frac-sym-uv}. 
Since $\uu = \Psi_p^r(\uu_0)$ and $\vv = \frac{q}{p} \Psi_p^r(\uu_0)$,
the conclusion follows.
\end{proof}

%%%%%%%%%%%%%%%%%%%%%%%%%%%%%%%%%%%%%%%%%%%%%%%%%%%%%%%%%%%%%%%%%%%%%%%%%%%%%%%%%%%
%
% Section 6: negative dilations: Necessity
%
%%%%%%%%%%%%%%%%%%%%%%%%%%%%%%%%%%%%%%%%%%%%%%%%%%%%%%%%%%%%%%%%%%%%%%%%%%%%%%%%%%%
\section{Negative dilations classification: necessity}
\label{sec:negative-N}

This section proves   the necessity  for membership in $S$ of the conditions in  Theorem \ref{thm:negative}.

%%%%%%%%%%%%%%%%%%%%
%
% Subsection 6.1   
%
%%%%%%%%%%%%%%%%%%%%%
\subsection{Birational Coordinate change: $(\tbetap, \trhop)$-coordinates }\label{sec:birational-2} 

We birationally transform  the parameter space  
to a third set of coordinates,   $(\tbetap, \trhop)$-coordinates,
given by
\begin{equation}
(\tbetap, \trhop) := \left( \frac{1}{\uu}, \frac{1}{\vv} \right) 
= \left( \betap, \frac{\betap}{\alphap} \right) 
\quad\text{where}\quad \alphap, \betap > 0.
\end{equation}
The inverse map is    
$(\alphap, \betap) = (\frac{\tbetap}{\trhop}, \tbetap)$.
The negative dilation part of the set $S$ viewed in $(\tbetap, \trhop)$-coordinates
is pictured in Figure \ref{fig:rhoprime-coord}.

%%%%%%%%%%%%%%%%%%%%%%%%%%%%%%
%
%  Figure N6.1 (\sigma-prime, \tau-prime)=(\tbetap, \trhop) coordinates 
%
%%%%%%%%%%%%%%%%%%%%%%%%%%%%%%%
\begin{figure}[h]
\begin{center}
 %%%%%%%%%%%
\begin{tikzpicture}[scale=.50]
  \draw[->] (0,0) -- (10,0) node[right] {$\tbetap=-\beta$};
  \draw[->] (0,0) -- (0,10) node[above] {$\trhop=\beta/\alpha$};
  \draw[scale=8.6] (1,0.05) -- (1,-0.05) node[below] {$1$};
  \draw[scale=8.6] (0.05,1) -- (-0.05,1) node[left] {$1$};
  
  \draw[scale=8.6,domain=0:1.1,smooth,variable=\x, blue] plot ({\x},{1});
  \draw[scale=8.6,domain=0:1.1,smooth,variable=\x, blue] plot ({\x},{1/2});
  \draw[scale=8.6,domain=0:1.1,smooth,variable=\x, blue] plot ({\x},{1/3});
  \draw[scale=8.6,domain=0:1.1,smooth,variable=\x, blue] plot ({\x},{1/4});
  \draw[scale=8.6,domain=0:1.1,smooth,variable=\x, blue] plot ({\x},{1/5});

  \draw[scale=8.6,domain=0:1.0,smooth,variable=\y, red]  plot ({\y},{\y});
  \draw[scale=8.6,domain=0:1.0,smooth,variable=\x, red]  plot ({\x},{\x/2});
  \draw[scale=8.6,domain=0:1.0,smooth,variable=\x, red]  plot ({\x},{\x/3});
  \draw[scale=8.6,domain=0:1.0,smooth,variable=\x, red]  plot ({\x},{\x/4});
  \draw[scale=8.6,domain=0:1.0,smooth,variable=\x, red]  plot ({\x},{\x/5});
  \draw[scale=8.6,domain=0:1.0,smooth,variable=\x, red]  plot ({\x/2},{\x});
  \draw[scale=8.6,domain=0:1.0,smooth,variable=\x, red]  plot ({\x/2},{\x/3});
  \draw[scale=8.6,domain=0:1.0,smooth,variable=\x, red]  plot ({\x/2},{\x/5});
  \draw[scale=8.6,domain=0:1.0,smooth,variable=\x, red]  plot ({\x/3},{\x});
  \draw[scale=8.6,domain=0:1.0,smooth,variable=\x, red]  plot ({\x/3},{\x/2});
  \draw[scale=8.6,domain=0:1.0,smooth,variable=\x, red]  plot ({\x/3},{\x/4});
  \draw[scale=8.6,domain=0:1.0,smooth,variable=\x, red]  plot ({\x/3},{\x/5});
  \draw[scale=8.6,domain=0:1.0,smooth,variable=\x, red]  plot ({\x/4},{\x});
  \draw[scale=8.6,domain=0:1.0,smooth,variable=\x, red]  plot ({\x/4},{\x/3});
  \draw[scale=8.6,domain=0:1.0,smooth,variable=\x, red]  plot ({\x/4},{\x/5});
  \draw[scale=8.6,domain=0:1.0,smooth,variable=\x, red]  plot ({\x/5},{\x});
  \draw[scale=8.6,domain=0:1.0,smooth,variable=\x, red]  plot ({\x/5},{\x/2});
  \draw[scale=8.6,domain=0:1.0,smooth,variable=\x, red]  plot ({\x/5},{\x/3});
  \draw[scale=8.6,domain=0:1.0,smooth,variable=\x, red]  plot ({\x/5},{\x/4});
  
%   %sporadic solutions
  \foreach \y in {2/3,3/5,4/7,5/9,6/11}{% node on the grid we have drawn 
    \node[draw,circle,inner sep=0.5pt,fill,red] at (2*\y*8.6,\y*8.6) {};  }
  \foreach \y in {2/5,3/8,4/11,5/14,6/17}{% node on the grid we have drawn 
    \node[draw,circle,inner sep=0.5pt,fill,red] at (3*\y*8.6,\y*8.6) {};  }
  
  \foreach \y in {9/7,9/8,12/11,15/14} {
    \node[draw,circle,inner sep=0.5pt,fill,red] at (1/2*\y*8.6,1/3*\y*8.6) {};}
  \foreach \y in {4/3,8/7,10/9,16/15} {
    \node[draw,circle,inner sep=0.5pt,fill,red] at (1/3*\y*8.6,1/2*\y*8.6) {};}
  
%   %%%%%%%%%%%%%%%%%%%   

  \draw[scale=8.6,domain=0:1,smooth,variable=\x, green] plot ({\x},{1-\x});
  \draw[scale=8.6,domain=0:0.5,smooth,variable=\x, green] plot ({\x},{1-2*\x});
  \draw[scale=8.6,domain=0:1,smooth,variable=\x, green] plot ({\x},{1/2-\x/2});
  \draw[scale=8.6,domain=0:0.33,smooth,variable=\x, green] plot ({\x},{1-3*\x});
  \draw[scale=8.6,domain=0:1,smooth,variable=\x, green] plot ({\x},{1/3-\x/3});
  \draw[scale=8.6,domain=0:0.5,smooth,variable=\x, green] plot ({\x},{1/2-\x});
\end{tikzpicture}
\end{center}
\caption{Negative dilation solutions of $S$  in  $(\tbetap, \trhop)$-coordinates: 
$\tbetap = -\beta,\, \trhop = \alpha/\beta$}
\label{fig:rhoprime-coord}
\end{figure}
%%%%%%%%%%%%%%%%%%%%%%%%%%%%%%
%
%  End Figure N6.1 (\sigma-prime, \tau-prime)=(\tbetap, \trhop) coordinates 
%
%%%%%%%%%%%%%%%%%%%%%%%%%%%%%%%

The $(\tbetap, \trhop)$ parameters take values in the positive quadrant.
The  negative dilation portion of $S$ viewed in  $(\tbetap, \trhop)$ coordinates 
consists of line segments and isolated points.

%%%%%%%%%%%%%%%%%%%%%%%%
%
% SUBSECTION 6.2 Completion of Proof of Sufficiency in Theorem 1.3
%
%%%%%%%%%%%%%%%%%%%%%%%%%%
\subsection{Necessity condition in $(\tbetap, \trhop)$-coordinates}
\label{sec:neg-necessity1}

We reformulate  the necessity condition in Theorem \ref{thm:negative} in terms of 
$(\tbetap, \trhop)$-coordinates. 

%%%%%%%%%%%%%%%%%%%
% Theorem 6.1 Necessary condition
%%%%%%%%%%%%%%%%%%%
\begin{thm}[Necessity in $(\tbetap, \trhop)$-coordinates]
\label{thm:neg-necessary}
For parameters $\tbetap, \trhop>0$, 
the nonnegative commutator relation ${ [f_{-\tbetap/\trhop}, f_{-\tbetap}] \geq 0 }$
does not hold unless one of the following conditions holds.

\begin{enumerate}[(a)]
\item
 If at least one of $\tbetap, \trhop$ is irrational, then  exactly one of the following two conditions holds.
\begin{enumerate}[(i)]
\item[$(i^{\ast})$] \label{it:hyperbola}
There are integers $m \geq 0,n\geq 1$ such that
\begin{equation*}
{m}\tbetap + {n}\trhop = 1.
\end{equation*}

\item[$(ii^{\ast})$] \label{it:line}
There are coprime integers $p,q$ such that
\begin{equation*}
p{\tbetap} = q{\trhop} \leq 1.
\end{equation*}

\end{enumerate}
\item If   $\tbetap, \trhop$ are both rational,
then: 
\begin{enumerate}
\item[(iii*)]
There are coprime integers $p,q$ and integers $m\geq 0, n\geq1$, and $r\geq 1$ such that
\begin{equation*}
{p}\tbetap = {q}\trhop = \left(1 + \frac{1}{r} \left( \frac{m}{p} + \frac{n}{q}-1\right)\right)^{-1}.
\end{equation*}
\end{enumerate}
\end{enumerate}
\end{thm}

The necessary conditions   in Theorem \ref{thm:neg-necessary}
 classify solutions into  three disjoint cases. 
This classification removes all rational solutions from cases $(i)$ and $(ii)$,
 giving cases $(i^{\ast})$ and $(ii^{\ast})$ respectively, with $(i^{\ast}) \subset (i)$ and $(ii^{\ast}) \subset (ii)$. 
One sees that $(i^{\ast})$  and $(ii^{\ast})$  are disjoint by  observing
that any common solution requires  solving  two linear equations in $(\tbetap, \trhop)$ with integer 
coefficients,
which  either are inconsistent or else have a
unique rational solution $(\tbetap, \trhop)$.

%%%%%%%%%%%%%%%%%%%%
%
% Subsection 6.3   Torus subgp criterion: negative dilations
%
%%%%%%%%%%%%%%%%%%%%%
\subsection{Torus subgroup criterion: negative dilations}
\label{sec:neg-disjoint}

To prove necessity of the classification given in Theorem \ref{thm:neg-necessary}
we establish a criterion for nonnegative commutator, given in terms of
$(\tbetap, \trhop)$-coordinates.
It is expressed in terms of a cyclic subgroup of the 2-dimensional torus\footnote{In the published
version of Part~I (\cite{LagR:2018a}), $\bT$ was denoted $\TT$.}
$$\bT := \RR^2/\ZZ^2$$
avoiding a set $\cCprimebar_{\tbetap, \trhop}$
which   also varies with the parameters. 
The set to be avoided, viewed in $\RR^2$,  
is an open rectangular region  with one corner at the origin,
modified to exclude a diagonal  
and to include two vertical sides of its boundary.
It is neither  open nor  closed as a subset of $\RR^2$.
 
Given a real number $x$, we let $\widetilde{x}$ denote its image under the quotient map $\RR \to \RR/\ZZ$.

%%%%%%%%%%%
% Defn 6.2 corner rectangle
%%%%%%%%%%%
\begin{defi}\label{defn:mod-corner-rectangle}
We define the {\em modified corner rectangle} 
$\cCprime_{\tbetap,\trhop}$ to be the    region
\begin{equation}\label{cCprime}
\cCprime_{\tbetap,\trhop} := 
\{(x,y) : 0\leq x \leq \tbetap, \, 0<y<\trhop,\, \frac{x}{\tbetap} \neq \frac{y}{\trhop} \} \subset \RR^2.
\end{equation}
We define $\cCprimebar_{\tbetap,\trhop}$ to be the image of 
$\cCprime_{\tbetap,\trhop}$ under coordinatewise projection $\RR^2 \to \bT$
to the torus ${ \bT = \RR^2 /\ZZ^2 }$.

We set $\cCprime = \cCprime_{1,1}$, so that 
\begin{equation}\label{cCprime11}
\cCprime := \cCprime_{1,1} = \{(x,y) : 0\leq x \leq 1, \, 0<y<1,\, x \neq y \} \subset \RR^2
\end{equation}
We call $\cCprime$ the {\em unit modified corner rectangle}.
\end{defi}

The  projected modified corner rectangle $\cCprimebar_{\tbetap, \trhop}$ consists of 
two connected components which are %half-open / half-closed 
triangles, as long as $0 < \tbetap, \trhop < 1$.
If $\tbetap \geq 1$ or $\trhop > 1$ then some ``wraparound'' occurs in the projection to the torus
$\RR^2 \to \bT= \RR^2 / \ZZ^2$.
See Figure \ref{fig:torus-mod-corner-rectangle} 
for examples of the projections $\cCprimebar_{\tbetap,\trhop}$.

%%%%%%%%%%%%%%%%%%%
%
% Figure N6.2Modified corner rectangle C-prime
%
%%%%%%%%%%%%%%%%%%%
\begin{figure}[h]
\centering 
\begin{center}
  \begin{tikzpicture}[scale=2.5]
  \draw[-] (4/9,0) -- (1,0);
  \draw[-] (0,0) -- (0,1);
  \draw[dotted] (1,0) -- (1,1);
  \draw[dotted] (0,1) -- (1,1);
  
  %generating point v
  \node[draw,circle,inner sep=0.5pt,fill,red] at (4/9,1/3) {};
  \draw[-] (4/9,1/3) node[right] {$(\tbetap, \trhop)$};
  
  %lower-left corner
  \foreach \x in {4/9} {
   \foreach \y in {1/3} {
    \fill[xscale=\x,yscale=\y,gray,nearly transparent] 
        (0,0+0.05) -- (0,1) -- (0.95,1) -- cycle;
    \fill[xscale=\x,yscale=\y,gray,nearly transparent] 
        (0+0.05,0) -- (1,0) -- (1,0.95) -- cycle;
    \draw[xscale=\x,yscale=\y,-] (0,0) -- (0,1);
    \draw[xscale=\x,yscale=\y,dashed] (0,1) -- (1,1);
    \draw[xscale=\x,yscale=\y,-] (1,1) -- (1,0);
    \draw[xscale=\x,yscale=\y,dashed] (1,0) -- (0,0);
    \draw[xscale=\x,yscale=\y,dashed] (0,0) -- (1,1);
  } }
  \end{tikzpicture}
\qquad\qquad
  \begin{tikzpicture}[scale=2.5]
  \draw[-] (2/5,0) -- (1,0);
  %\draw[-] (0,0) -- (0,1);
  \draw[dotted] (1,0) -- (1,1);
  \draw[dotted] (0,1) -- (1,1);
  
  %generating point v
  \node[draw,circle,inner sep=0.5pt,fill,red] at (2/5,6/5) {};
  \draw[-] (2/5,6/5) node[right] {$(\tbetap, \trhop)$};
  
  %lower-left corner
  \foreach \x in {2/5} {
   \foreach \y in {6/5} {
    \fill[xscale=\x,yscale=\y,gray,nearly transparent] 
        (1/6-0.05,1/6) -- (0,1/6) -- (0,1) -- (0.95,1) -- cycle;
    \fill[xscale=\x,yscale=\y,gray,nearly transparent] 
        (1/6+0.05,1/6) -- (1,1/6) -- (1,0.95) -- cycle;
    \draw[xscale=\x,yscale=\y,-] (0,0) -- (0,1);
    \draw[xscale=\x,yscale=\y,dashed] (0,1) -- (1,1);
    \draw[xscale=\x,yscale=\y,-] (1,1) -- (1,0);
    \draw[xscale=\x,yscale=\y,dashed] (1,0) -- (0,0);
    \draw[xscale=\x,yscale=\y,dashed] (0,0) -- (1,1);
  } }
  %wrap-around pieces
  \fill[gray,nearly transparent] (0,0) -- (0,1/5) -- (2/5,1/5) -- (2/5,0) --cycle;
  \fill[gray,nearly transparent] (0,0) -- (0,1/5) -- (2/5,1/5) -- (2/5,0) --cycle;
  \fill[gray,nearly transparent] (0,1) -- (0,6/5) -- (2/5,6/5) -- (2/5,1) --cycle;
  %\fill[gray,nearly transparent] (0,1) -- (0,6/5) -- (2/5,6/5) -- (2/5,1) --cycle;
  \draw[dashed] (2/5-1/15,0) -- (2/5,1/5);
  \end{tikzpicture}
\qquad\qquad
  \begin{tikzpicture}[scale=2.5]
  %\draw[-] (0,0) -- (1,0);
  %\draw[-] (0,0) -- (0,1);
  \draw[dotted] (1,0) -- (1,1);
  \draw[dotted] (0,1) -- (1,1);
  
  %generating point v
  \node[draw,circle,inner sep=0.5pt,fill,red] at (8/7,6/5) {};
  \draw[-] (8/7,6/5) node[right] {$(\tbetap, \trhop)$};
  
  %lower-left corner
  \foreach \x in {8/7} {
   \foreach \y in {6/5} {
    \fill[xscale=\x,yscale=\y,gray,nearly transparent] 
        (1/8,1/8+0.03) -- (1/8,1) -- (0.97,1) -- cycle;
    \fill[xscale=\x,yscale=\y,gray,nearly transparent] 
        (1/6+0.03,1/6) -- (1,1/6) -- (1,0.97) -- cycle;
    \draw[xscale=\x,yscale=\y,-] (0,0) -- (0,1);
    \draw[xscale=\x,yscale=\y,dashed] (0,1) -- (1,1);
    \draw[xscale=\x,yscale=\y,-] (1,1) -- (1,0);
    \draw[xscale=\x,yscale=\y,dashed] (1,0) -- (0,0);
    \draw[xscale=\x,yscale=\y,dashed] (0,0) -- (1,1);
  %wrap-around piecies
  \fill[xscale=\x,yscale=\y,gray,nearly transparent] (0,0) -- (1/8,0) -- (1/8,1) -- (0,1) -- cycle;
  \fill[xscale=\x,yscale=\y,gray,nearly transparent] (0,0) -- (1/8,0) -- (1/8,1) -- (0,1) -- cycle;
  \draw[dashed] (0,6/5-3/20) -- (1/7,6/5);
  \fill[xscale=\x,yscale=\y,gray,nearly transparent] (0,0) -- (0,1/6) -- (1,1/6) -- (1,0) -- cycle;
  \fill[xscale=\x,yscale=\y,gray,nearly transparent] (1/8,0) -- (1/8,1/6) -- (1,1/6) -- (1,0) -- cycle;
  \draw[dashed] (8/7-4/21,0) -- (8/7,1/5);
  \fill[xscale=\x,yscale=\y,gray,nearly transparent] (7/8,0) -- (1,0) -- (1,1) -- (7/8,1) -- cycle;
  %\fill[xscale=\x,yscale=\y,gray,nearly transparent] (7/8,0) -- (1,0) -- (1,1) -- (7/8,1) -- cycle;
  %\draw[dashed] (0,6/5-3/20) -- (1/7,6/5);
  \fill[xscale=\x,yscale=\y,gray,nearly transparent] (0,5/6) -- (0,1) -- (1,1) -- (1,5/6) -- cycle;
  %\fill[xscale=\x,yscale=\y,gray,nearly transparent] (1/8,5/6) -- (1/8,1) -- (1,1) -- (1,5/6) -- cycle;
  %\draw[dashed] (8/7-4/21,0) -- (8/7,1/5);
  } }
  \end{tikzpicture}
\end{center}
\caption{Modified corner rectangle $\cCprimebar_{\tbetap, \trhop}$ in the torus $\bT= \RR^2/\ZZ^2$; 
with various $(\tbetap, \trhop)$.}
\label{fig:torus-mod-corner-rectangle}
\end{figure} 
%%%%%%%%%%%%%%%%%%%
%
% END Figure 6-3 Modified corner rectangle C-prime
%
%%%%%%%%%%%%%%%%%%%

%%%%%%%%%%%%%%%%%%%
%
% Proposition/ N6.2 Torus Subgroup criterion (NEGATIVE DILATIONS)
%
%%%%%%%%%%%%%%%%%%%
\begin{prop}[Torus subgroup criterion-negative dilations]
\label{lem:torus-neg}
For $\tbetap,\trhop>0$, the following conditions are equivalent.
\begin{enumerate}
\item[\textnormal{(Q1')}] 
The nonnegative commutator relation holds:
\[ 
[f_{-\tbetap/\trhop}, f_{-\tbetap}] (x) \geq 0 \quad \mbox{for all} \quad x \in \RR.
 \]

\item[\textnormal{(Q2')}] 
The cyclic torus subgroup 
\[
\langle (\widetilde{\tbetap},\widetilde{\trhop}) \rangle_{\bT} 
= {\{ (\widetilde{n\tbetap}, \widetilde{n\trhop} ) : n\in \ZZ \}}  \subset \RR^2/\ZZ^2=
\bT
\] 
is disjoint from the modified corner rectangle  
\[
\cCprimebar_{\tbetap,\trhop} = \{(\widetilde{x},\widetilde{y}) : 0\leq x \leq \tbetap, 
\, 0<y<\trhop,\, \frac{x}{\tbetap} \neq \frac{y}{\trhop} \} 
\subset \bT .
\]
\end{enumerate}
\end{prop}

%%%%%%%%%%%%%%%%%%%
% Proof of Proposition/  N6.2
%%%%%%%%%%%%%%%%%%%
\begin{proof}
The proof follows  the same argument as that of 
\cite[Proposition 6.2]{LagR:2018a}, replacing $(\talpha,\trho)$ with $(\tbetap, \trhop)$ and replacing 
\cite[Proposition 5.2]{LagR:2018a}  with Proposition \ref{prop:diagonal-neg}.
\end{proof}

%%%%%%%%%%%%%%%%%%%%
%
% Subsection 6.4   Proof of necessity Proposition part (a) 
%
%%%%%%%%%%%%%%%%%%%%%

%%%%%%%%%%%%%%%%%%%
% SUB-SECTION N6.2.1 Proof of Necessity Prop. 8.3
%%%%%%%%%%%%%%%%%%%
\subsection{Proof of  Theorem \ref{thm:neg-necessary}}
\label{sec:neg-necess-proof}

We first recall a technical result needed  in the following proofs.

%%%%%%%%%%%%%%%%%%%
% Theorem 6.4 Closed subgroup theorem
%%%%%%%%%%%%%%%%%%%
\begin{thm}[Closed subgroup theorem]
\label{thm:lie-subgroup}
Given a Lie group $G$ and a subgroup $H\subset G$, 
the topological closure $\bar{H}$ of the subspace $H$ is a Lie subgroup $\bar{H}\subset G$.
\end{thm}
\begin{proof}
See  Lee \cite[Theorem 20.12, p. 523]{Lee13}.
Under these hypotheses, $\bar{H}$ is a closed subgroup of $G$.
\end{proof}

%%%%%%%%%%%%%%%%%%%
% Proof of Theorem 6.1
%%%%%%%%%%%%%%%%%%%

\begin{proof}[Proof of Theorem \ref{thm:neg-necessary} (a)]
Suppose that at least one of $\tbetap, \trhop$ is irrational. 
Then the map
\begin{align*}
\widetilde{\phi}_{\tbetap, \trhop} : \ZZ &\to \bT \\
k &\mapsto (\widetilde{k\tbetap}, \widetilde{k\trhop})
\end{align*}
is injective and the subgroup $H = \langle (\widetilde{\tbetap},\widetilde{\trhop}) \rangle_{\bT}$ 
of the torus is infinite cyclic. 
Since $\bT$ is compact, $H$ cannot be discrete.
Thus the closure $\bar{H}$ of this subgroup  must be a Lie subgroup $\bar{H}$ of dimension 1 or 2, by 
the closed subgroup theorem for Lie groups
\cite[Theorem 20.12, p. 523]{Lee13}.

{\em Case 1.} If $\bar{H}$ is dimension $2$ then $H$ is dense in $\bT$
and will intersect the non-empty region $\cCprimebar_{\tbetap,\trhop}$. 
So condition (Q2') of  Proposition \ref{lem:torus-neg} is not satisfied, hence  
the nonnegative commutator relation (Q1') does not hold.

{\em Case 2.} If $\bar{H}$  is dimension $1$, then the parameters $\tbetap,\trhop$ must satisfy 
an  integer relation of the form
\begin{equation*}
m\tbetap + n\trhop = k, \quad m,n,k \in \ZZ ,\quad m, n,k\text{ coprime}%,\, k\neq 0
\end{equation*}
and the subgroup $\bar{H}$ is the projection to $\bT$ of the lines
 $\{ (x,y)\in \RR^2 : mx + ny \in \ZZ \}$.
 In the cases below, we will find solutions only when $k=0$, giving case $(ii^{\ast})$,
 or $k = \pm 1$, giving case $(i^{\ast})$. 

{\em Case 2a-1.} If the integers $m,n$ are nonzero and have opposite sign, 
then the lines in $\bar{H}$ have positive slope. 
If in addition $k\neq 0$ then the lines in $\bar{H}$ 
will have different slope ($ = -\frac{m}{n}$) than 
the punctured diagonal of $\cCprime_{\tbetap,\trhop}$ (slope $= \frac{\trhop}{\tbetap}$), 
so the  intersection $\bar{H}\cap \cCprime_{\tbetap,\trhop}$ 
 will be non-empty in a neighborhood of $(0,0)$.
See Figure \ref{fig:torusprime-pos-slope}.
Thus the commutator relation $[f_{-\tbetap/\trhop}, f_{-\tbetap}]\geq 0$ does 
not hold by Proposition \ref{lem:torus-neg}.

{\em Case 2a-2.} If $m,n$ have opposite sign and $k = 0$, 
then the assumption $\gcd(m,n,k) = 1$ implies $m,n$ are coprime.
Set $p=|m|$, $q=-|n|$ so that $p,q$ coprime positive integers, with $p\tbetap - q\trhop =0$ and 
$\bar{H} = \{(x,y) : px - qy \in \ZZ\}$.
$\bar{H}$ will not  intersect $\cCprime_{\tbetap,\trhop}$ in a neighborhood of $(0,0)$, since the segment of $\bar{H}$ at $(0,0)$ lies precisely in the diagonal  of $\cCprime_{\tbetap,\trhop}$. 
See Figure \ref{fig:torusprime-pos-slope}.

The ``modified corner rectangle'' $\cCprime_{\tbetap,\trhop}$ will then be disjoint from the subgroup $\bar{H}$ if and only if the rectangle 
is contained in the open region $\{(x,y) : |px-qy| < 1 \}$.
This containment occurs if and only if the extremal corners 
$(\tbetap, 0)$ and $(0, \trhop)$ lie in the closed region $\{ (x,y) : |px - qy| \leq 1\}$. 
(These corners are in the closure of $\cCprime_{\tbetap,\trhop}$, 
but not in $\cCprime_{\tbetap,\trhop}$ itself.)
This means $p\tbetap = q\trhop \leq 1$, so we are in case $(ii^{\ast})$  of the classification.

%%%%%%%%%%%%%%%%%%%
% Figure 6.3 [formerly fig 9]. 2x=3y in Z comparison
%%%%%%%%%%%%%%%%%%%

\begin{figure}[h]
\begin{center}
\begin{tikzpicture}[scale=3]
  %boundary of square torus
  \draw[-] (0,0) -- (1,0);
  \draw[-] (0,0) -- (0,1);
  \draw[dotted] (1,0) -- (1,1);
  \draw[dotted] (0,1) -- (1,1);
  
  %diagonal lines
  \draw[domain=0:1,smooth,variable=\x, blue] plot ({\x},{2/3*\x});
  \draw[domain=0:1/2,smooth,variable=\x, blue] plot ({\x},{2/3*\x+2/3});
  \draw[domain=1/2:1,smooth,variable=\x, blue] plot ({\x},{2/3*\x-1/3});
  \draw[domain=0:1,smooth,variable=\x, blue] plot ({\x},{2/3*\x+1/3});
  %\draw[domain=2/3:1,smooth,variable=\x, blue]  plot ({\x},{2/3*\x-1/2});
  %\draw[domain=0:1,smooth,variable=\x, blue] plot ({\x},{2/3*\x+1/4});
  
  %lower-left corner
  \foreach \n in {0} {
    \fill[xscale=2/5,yscale=0.6,gray,nearly transparent] 
        (\n,\n+0.05) -- (\n,\n+1) -- (\n+0.95,\n+1) -- cycle;
    \fill[xscale=2/5,yscale=0.6,gray,nearly transparent] 
        (\n+0.05,\n) -- (\n+1,\n) -- (\n+1,\n+0.95) -- cycle;
    \draw[xscale=2/5,yscale=9/15,-] (\n,\n) -- (\n,\n+1);
    \draw[xscale=2/5,yscale=9/15,dashed] (\n,\n+1) -- (\n+1,\n+1);
    \draw[xscale=2/5,yscale=9/15,-] (\n+1,\n+1) -- (\n+1,\n);
    \draw[xscale=2/5,yscale=9/15,dashed] (\n+1,\n) -- (\n,\n);
    \draw[xscale=2/5,yscale=9/15,dashed] (\n,\n) -- (\n+1,\n+1);
  }
  
  %generating point v
  \node[draw,circle,inner sep=0.5pt,fill,red] at (2/5,9/15) {};
  \draw[-] (2/5,9.5/15) node[above] {$(\tbetap,\trhop)$};
    
  \end{tikzpicture}
%%%%%%%%%%%
\qquad \qquad
%%%%%%%%%%%
\begin{tikzpicture}[scale=3]
  %boundary of square torus
  \draw[-] (0,0) -- (1,0);
  \draw[-] (0,0) -- (0,1);
  \draw[dotted] (1,0) -- (1,1);
  \draw[dotted] (0,1) -- (1,1);
  
  %diagonal lines
  \draw[domain=0:1,smooth,variable=\x, blue] plot ({\x},{2/3*\x});
  \draw[domain=0:1/2,smooth,variable=\x, blue] plot ({\x},{2/3*\x+2/3});
  \draw[domain=1/2:1,smooth,variable=\x, blue] plot ({\x},{2/3*\x-1/3});
  \draw[domain=0:1,smooth,variable=\x, blue] plot ({\x},{2/3*\x+1/3});
  %\draw[domain=2/3:1,smooth,variable=\x, blue]  plot ({\x},{2/3*\x-1/2});
  %\draw[domain=0:1,smooth,variable=\x, blue] plot ({\x},{2/3*\x+1/4});
  
  %lower-left corner
  \foreach \n in {0} {
    \fill[xscale=2/5,yscale=0.26,gray,nearly transparent] 
        (\n,\n+0.08) -- (\n,\n+1) -- (\n+0.90,\n+1) -- cycle;
    \fill[xscale=2/5,yscale=0.26,gray,nearly transparent] 
        (\n+0.08,\n) -- (\n+1,\n) -- (\n+1,\n+0.92) -- cycle;
    \draw[xscale=2/5,yscale=4/15,-] (\n,\n) -- (\n,\n+1);
    \draw[xscale=2/5,yscale=4/15,dashed] (\n,\n+1) -- (\n+1,\n+1);
    \draw[xscale=2/5,yscale=4/15,-] (\n+1,\n+1) -- (\n+1,\n);
    \draw[xscale=2/5,yscale=4/15,dashed] (\n+1,\n) -- (\n,\n);
    \draw[xscale=2/5,yscale=4/15,dashed] (\n,\n) -- (\n+1,\n+1);
  }
  
  %generating point v
  \node[draw,circle,inner sep=0.5pt,fill,red] at (2/5,4/15) {};
  \draw[-] (2/5,4.5/15) node[above] {$(\tbetap,\trhop)$};
    
  \end{tikzpicture}
\end{center}
\caption{Solutions to $2x-3y\in\ZZ$ compared to $\cCprime_{\tbetap,\trhop}$, when $k\neq 0$ and $k=0$.}
\label{fig:torusprime-pos-slope}
\end{figure} 
%%%%%%%%%%%%%%%%%%%
% END Figure 6.3 [formerly fig 9]. 2x=3y in Z comparison
%%%%%%%%%%%%%%%%%%%

{\em Case 2b.} Suppose  $m,n$ do not have opposite sign so we may assume 
that $m,n$ are  nonnegative without loss of generality.
Then $k\geq 1$ since $\tbetap, \trhop >0$.

{\em Case 2b-1.} If $n = 0$, then $\bar{H}$ consists of vertical lines 
$\{(x,y) : mx \in \ZZ \}$ 
which intersect the vertical boundary segments of $\cCprime_{\tbetap,\trhop}$. 
Thus the commutator relation
$[f_{-\tbetap/\trhop}, f_{-\tbetap}]\geq 0$ does not hold 
by Proposition \ref{lem:torus-neg}.

{\em Case 2b-2.} Suppose $n \ge 1$ and $k\geq 2$. 
Then the point $(\frac{1}{k}\tbetap, \frac{1}{k}\trhop)$ 
lies in the closed subgroup $\bar{H}$,
and $\bar{H}$ intersects the modified corner rectangle $\cCprime_{\tbetap, \trhop}$, 
in any open neighborhood of this point. 
See Figure \ref{fig:torusprime-neg-slope}.
Thus $H$ also intersects $\cCprime_{\tbetap, \trhop}$ 
and the commutator relation $[f_{-\tbetap/\trhop}, f_{-\tbetap}]\geq 0$ does not hold
by Proposition \ref{lem:torus-neg}.

{\em Case 2b-3.} Suppose $n\geq 1$ and $k=1$. 
Then the closed subgroup $\bar{H}$ does not intersect 
the modified corner rectangle $\cCprime_{\tbetap, \trhop}$ 
because $\cCprime_{\tbetap, \trhop}$ lies in the region 
$\{ (x,y) : 0 < mx + ny < 1\}$;
see Figure \ref{fig:torusprime-neg-slope}.
This subcase is case 
$(i^{\ast})$ of the classification.

%%%%%%%%%%%%%%%%%%%
% Figure 6.4  3x+2y \in Z comparison
%%%%%%%%%%%%%%%%%%%

\begin{figure}[h]
\begin{center}
  \begin{tikzpicture}[scale=3]
  \draw[-] (0,0) -- (1,0);
  \draw[-] (0,0) -- (0,1);
  \draw[dotted] (1,0) -- (1,1);
  \draw[dotted] (0,1) -- (1,1);
  \draw[domain=0:2/3,smooth,variable=\x, blue] plot ({\x},{1-3/2*\x});
  \draw[domain=2/3:1,smooth,variable=\x, blue] plot ({\x},{2-3/2*\x});
  \draw[domain=0:1/3,smooth,variable=\x, blue] plot ({\x},{1/2-3/2*\x});
  \draw[domain=1/3:1,smooth,variable=\x, blue] plot ({\x},{3/2-3/2*\x});
  
  %lower-left corner
  \foreach \n in {0} {
    \fill[xscale=4/9,yscale=1/3,gray,nearly transparent] (\n,\n+0.05) -- (\n,\n+1) -- (\n+0.95,\n+1) -- cycle;
    \fill[xscale=4/9,yscale=1/3,gray,nearly transparent] (\n+0.05,\n) -- (\n+1,\n) -- (\n+1,\n+0.95) -- cycle;
    \draw[xscale=4/9,yscale=1/3,-] (\n,\n) -- (\n,\n+1);
    \draw[xscale=4/9,yscale=1/3,dashed] (\n,\n+1) -- (\n+1,\n+1);
    \draw[xscale=4/9,yscale=1/3,-] (\n+1,\n+1) -- (\n+1,\n);
    \draw[xscale=4/9,yscale=1/3,dashed] (\n+1,\n) -- (\n,\n);
    \draw[xscale=4/9,yscale=1/3,dashed] (\n,\n) -- (\n+1,\n+1);
  }
  
  %generating point v
  \node[draw,circle,inner sep=0.9pt,fill,red] at (4/9,1/3) {};
  \draw[-] (4/9,1/3) node[right] {$(\tbetap, \trhop)$};
  
  \node[draw,circle,inner sep=0.9pt,fill,red] at (2/9,1/6) {};
  \draw[-] (2/10,1/6) node[left] {$(\frac{1}{2} \tbetap, \frac{1}{2} \trhop)$};
\end{tikzpicture}
%%%%%%%%%%%%%%%%%%%
\qquad\qquad
%%%%%%%%%%%%%%%%%%%
  \begin{tikzpicture}[scale=3]
  \draw[-] (0,0) -- (1,0);
  \draw[-] (0,0) -- (0,1);
  \draw[dotted] (1,0) -- (1,1);
  \draw[dotted] (0,1) -- (1,1);
  \draw[domain=0:2/3,smooth,variable=\x, blue] plot ({\x},{1-3/2*\x});
  \draw[domain=2/3:1,smooth,variable=\x, blue] plot ({\x},{2-3/2*\x});
  \draw[domain=0:1/3,smooth,variable=\x, blue] plot ({\x},{1/2-3/2*\x});
  \draw[domain=1/3:1,smooth,variable=\x, blue] plot ({\x},{3/2-3/2*\x});
  
  %lower-left corner
  \foreach \n in {0} {
    \fill[xscale=2/9,yscale=1/6,gray,nearly transparent] (\n,\n+0.05) -- (\n,\n+1) -- (\n+0.95,\n+1) -- cycle;
    \fill[xscale=2/9,yscale=1/6,gray,nearly transparent] (\n+0.05,\n) -- (\n+1,\n) -- (\n+1,\n+0.95) -- cycle;
    \draw[xscale=2/9,yscale=1/6,-] (\n,\n) -- (\n,\n+1);
    \draw[xscale=2/9,yscale=1/6,dashed] (\n,\n+1) -- (\n+1,\n+1);
    \draw[xscale=2/9,yscale=1/6,-] (\n+1,\n+1) -- (\n+1,\n);
    \draw[xscale=2/9,yscale=1/6,dashed] (\n+1,\n) -- (\n,\n);
    \draw[xscale=2/9,yscale=1/6,dashed] (\n,\n) -- (\n+1,\n+1);
}  
  %generating point v
  \node[draw,circle,inner sep=0.9pt,fill,red] at (2/9,1/6) {};
  \draw[-] (2/9,1/6) node[right] {$(\tbetap, \trhop)$};
\end{tikzpicture}
\end{center}
\caption{Solutions to $3x+2y\in\ZZ$ in the torus $\bT$; with $k=2$ and $k=1$}
\label{fig:torusprime-neg-slope}
\end{figure} 
%%%%%%%%%%%%%%%%%%%
% END Figure 6.4  3x+2y \in Z comparison
%%%%%%%%%%%%%%%%%%%

Theorem \ref{thm:neg-necessary} (a)  is established.
\end{proof}

%%%%%%%%%%%%%%%%%%%
% Proof of Necessity Theorem N6.4(b)
%%%%%%%%%%%%%%%%%%%

To prove  Theorem \ref{thm:neg-necessary} (b),
we mainly use $(\uu,\vv)$-coordinates rather than $(\tbetap, \trhop)$-coordinates.

\begin{proof}[Proof of Theorem \ref{thm:neg-necessary} (b)]
If $\tbetap,\trhop$ are both rational, then 
$H = \langle (\widetilde{\tbetap},\widetilde{\trhop})\rangle_{\bT}$ 
is a discrete subgroup of the torus.
By Lemma \ref{prop:diagonal-neg} the commutator inequality 
$[f_{-\tbetap/\trhop},f_{-\tbetap}]\geq 0$ means $\Lambda_{\uu,\vv}$ 
is 
disjoint from $\cDprime$, 
where $\uu = \tbetap^{-1}$ and $\vv = \trhop^{-1}$ and
$ \Lambda_{\uu,\vv} = \rowspan_\ZZ\begin{pmatrix}
\uu & 0  \\
0 & \vv 
\end{pmatrix}$.

Let $\lambda = \egcd(\uu,\vv)$, so that $(\uu, \vv) = (\lambda p, \lambda q)$ for coprime integers $p,q$.
By  
Lemma \ref{lem:neg-sporadic-converse}, there exists an integer $r\geq 1$ such that
the lattice
\[ \Lambda_{\ell p,\ell q} = \rowspan_\ZZ\begin{pmatrix}
\ell p & 0  \\
0 & \ell q 
\end{pmatrix}
\quad\text{with}\quad
 \ell =  1 + r(\lambda -1) \]
is disjoint from $\cD\cup\cDprime$.
This means $\Lambda_{\ell p,\ell q}$  must be of the form in Lemma \ref{lem:diagonal-vert-boundary}:
there exist integers $m\geq 0, n\geq 1$ such that $(\ell p, \ell q)$ lies on the hyperbola
\[ 
\{ (x,y) \in \RR^2 : \frac{m}{x} + \frac{n}{y} = 1 \}
\]
so we must have
$ \ell = \frac{m}{p}+\frac{n}{q} $.
The relation $\ell  = 1 + r(\lambda - 1)$
is equivalent to $\lambda = 1 + \frac{1}{r}(\ell - 1)$.
Thus we have
\[ 
\frac{\uu}{p} = \frac{\vv}{q} = \lambda 
= 1  + \frac1{r}\left(\frac{m}{p}+\frac{n}{q} - 1\right).
\]
Finally,  we recall that $\tbetap = \uu^{-1}$ and $\trhop = \vv^{-1}$
so taking reciprocals of the above equalities gives
\[ p\tbetap = q\trhop = \left( 1 + \frac1{r} \left(\frac{m}{p}+\frac{n}{q}\right) - 1 \right)^{-1} .\]
 Theorem \ref{thm:neg-necessary} (b)  is established.
\end{proof}
%This solution falls in  case $(iii)$ if $r\geq 2$.
%If $r = 1$, it falls in  case $(i^*)$.

Theorem \ref{thm:negative} is now proved,  combining
Theorem \ref{thm:neg-sufficient} and Theorem \ref{thm:neg-necessary}.

%%%%%%%%%%%%%%%%%%%%%%%%%
%
% Section 7 Proof of Closure property
%
%%%%%%%%%%%%%%%%%%%%%%%%%
\section{Proof of Closure Theorem \ref{thm:closed} for negative dilations}
\label{sec:consequences}

%%%%%%%%%%%%%%%%%%%
% Theorem 1.4 Closed subset
%%%%%%%%%%%%%%%%%%%
%\begin{customthm}{1.4}[Closed Set Property of $S$]
%\label{thm:closed-N}
%The set $S$ of all pairs of dilation factors $(\alpha, \beta)$ which satisfy the nonnegative commutator inequality ${[f_\alpha, f_\beta]\geq 0 }$, where $f_\alpha(x) = \floor{\alpha x}$,
%is a closed subset of $\RR^2$.
%\end{customthm}

%%%%%%%%%%%%%%%%%%%
% Proof of Theorem 1.2
%%%%%%%%%%%%%%%%%%%
\begin{proof} 
Let $\RR^2_{-} = \{ (\alpha,\beta) \in \RR^2 : \alpha, \beta < 0 \}$
denote the open third quadrant,
and let 
$
S_{-} = S\cap \Qthree = \{ (\alpha,\beta) \in \Qthree : %\alpha,\beta <0,\, 
[f_\alpha,f_\beta]\geq 0\}$.
Let $\Si$, $\Sii$, and $\Siii$  
denote the subsets of $S_{-}$
corresponding to cases $(i)$, $(ii)$, and $(\text{iii}^{\ast})$ of Theorem~\ref{thm:negative}.
It is straightforward to check that  $\Si$ and $\Sii$ are closed in $\RR^2_{-}$.
The set $\Siii$  contains the ``sporadic rational solutions'';
 by itself,
$\Siii$ does not form a closed set in $\RR^2_{-}$.

{\bf Claim 1.} {\em If a sequence $(\alpha_i,\beta_i)$ of points in $\Siii \smallsetminus \Si$ 
has a limit point in the open third quadrant,
then the values $\alpha_i$ are eventually constant.}

Proof of Claim 1: 
Suppose $(\alpha_i, \beta_i)$ is a sequence in $\Siii \smallsetminus \Si $ which 
 converges to a limit $(\alpha, \beta)$ in the open third quadrant.
Since the coordinate $ \beta$ is strictly negative, there is some positive integer $N$ such that 
$\beta < - \frac{1}{N}$.
This implies  $\beta_i < -\frac{1}{N}$
for all sufficiently large $i$.
Suppose $\alpha_i = - \frac{q_i}{p_i}$ as a reduced fraction.
Case $(iii^*)$ of Theorem \ref{thm:negative} says that
\begin{equation}
\label{eq:beta}
 \beta_i = -\frac{1}{p_i} \left(1 + \frac1{r_i} \left(\frac{m_i}{p_i} + \frac{n_i}{q_i} - 1 \right) \right)^{-1}
\end{equation}
for some integers $m_i \geq 0$, $n_i \geq 1$, and $r_i \geq 1$;
since the point is not in case $(i)$, we have $r_i \geq 2$.
From these conditions it follows that 
$-\frac{2}{p_i} < \beta_i $.
Since eventually $\beta_i < -\frac{1}{N}$, these bounds 
 imply that $\frac1{p_i} > \frac1{2N}$ %$p_i < 2N$ 
 for sufficiently large $i$.
This means that for sufficiently large $i,j$, if $\alpha_i \neq \alpha_j$ then 
$ 
|\alpha_i - \alpha_j | = \left|\frac{q_i}{p_i} - \frac{q_j}{p_j} \right| 
> \frac{1}{4N^2} 
.$
Thus our assumption that the sequence $(\alpha_i, \beta_i)$ converges implies that $\alpha_i$ 
must be eventually constant,
which proves  Claim 1.

{\bf Claim 2.} 
{\em The set of points in  $\Siii \smallsetminus \Si $ with fixed $\alpha= -\frac{q}{p}$  
has all its limit points in $\Sii$. }

\noindent This claim is straightforward to verify using \eqref{eq:beta}, with $p_i = p$ and $q_i = q$ fixed.

The  two claims above imply that all limit points of $\Siii$ are  in 
$\Si$ or $\Sii$,
so the union 
${S}_{-} = \Si \cup  \Sii \cup  \Siii$ 
is a closed subset of $\RR^2_{-}$.
\end{proof}

\section*{Acknowledgements}
The second author thanks David Speyer and  Yilin Yang for helpful conversations. 
%Work of the  first author was partially supported by NSF grant DMS-1401224 and DMS-1701576,
%by MSRI as a Chern Professor and by a Simons  Fellowship in Mathematics.
%%MSRI is supported in part by the National Science Foundation. 
%Work of the second author was partially supported by NSF grant DMS-1600223.

%%%%%%%%
% REFERENCES
%%%%%%%%

\end{document}